%% file: main.tex
\documentclass[a4paper,11pt]{article} 
\input{preamble}

\title{%
    The Mean Field Ensemble Kalman Filter: Near-Gaussian Setting
}
\author{%
    J. A. Carrillo\thanks{%
        Mathematical Institute, University of Oxford, UK (\email{carrillo@maths.ox.ac.uk}).
    } \and
    F. Hoffmann\thanks{%
        Department of Computing and Mathematical Sciences, Caltech, USA (\email{franca.hoffmann@caltech.edu}).
    } \and
    A. M. Stuart\thanks{%
        Department of Computing and Mathematical Sciences, Caltech, USA (\email{astuart@caltech.edu}).
    } \and
    U. Vaes\thanks{%
        MATHERIALS, Inria Paris \& CERMICS, \'Ecole des Ponts, France (\email{urbain.vaes@inria.fr})
    }
}

\begin{document}
\maketitle

\begin{abstract}
    The ensemble Kalman filter is widely used in applications because, for high dimensional filtering problems, it has a robustness that is not shared for example by the particle filter; in particular it does not suffer from weight collapse. However, there is no theory which quantifies its accuracy as an approximation of the true filtering distribution, except in the Gaussian setting. To address this issue we provide the first analysis of the accuracy of the ensemble Kalman filter beyond the Gaussian setting. We prove two types of results: the first type comprise a stability estimate controlling the error made by the ensemble Kalman filter in terms of the difference between the true filtering distribution and a nearby Gaussian; and the second type use this stability result to show that, in a neighbourhood of Gaussian problems, the ensemble Kalman filter makes a small error, in comparison with the true filtering distribution. Our analysis is developed for the mean field ensemble Kalman filter. We rewrite the update equations for this filter, and for the true filtering distribution, in terms of maps on probability measures. We introduce a weighted total variation metric to estimate the distance between the two filters and we prove various stability estimates for the maps defining the evolution of the two filters, in this metric. Using these stability estimates  we prove results of the first and second types, in the weighted total variation metric. We also provide a generalization of these results to the Gaussian projected filter, which can be viewed as a mean field description of the unscented Kalman filter.
\end{abstract}



\section{Introduction}
\label{sec:I}

\subsection{Context}
\label{ssec:C}

This paper is concerned with the study of partially and noisily observed
dynamical systems.
Filtering refers to the sequential updating of the probability distribution
of the state of the (possibly stochastic) dynamical system, given partial noisy
observations \cite{asch2016data,MR3363508,reich2015probabilistic}.
The Kalman filter, introduced in 1960,
provides an explicit solution to this problem in the setting of
linear dynamical systems, linear observations and additive Gaussian noise \cite{kalman1960new}; the desired probability distribution is Gaussian and the
Kalman filter provides explicit update formulae for the mean and covariance.
The extended Kalman filter is a linearization-based methodology which was developed
in the 1960s and 1970s to apply to situations beyond the linear-Gaussian setting \cite{jazwinski2007stochastic}. It is, however, not practical in high dimensions, because of the need to compute, and sequentially update,
large covariance matrices~\cite{ghil}. To address this issue, the ensemble Kalman filter was introduced in 1994 \cite{evensen1994sequential}, using ensemble-based low-rank approximations of the covariances, and is a method well-adapted to high dimensional problems. The unscented Kalman filter, introduced in 1997, provides an alternative approach to nonlinear and non-Gaussian problems \cite{julier1997new}, using quadrature to approximate covariance matrices, and is well-adapted to problems of moderate dimension.
The particle filter \cite{doucet2001sequential} is a provably convergent methodology for approximating the filtering distribution \cite{cricsan1999interacting,MR3375889}.
However it does not scale well to high dimensional problems \cite{MR2459233,snyder2008obstacles}; this motivates the increasing adoption of
ensemble Kalman methods.

Over the last two decades ensemble Kalman filters have found widespread
use in the geophysical sciences,
are starting to be used in other application domains, and have been developed as a general purpose
tool for solving inverse problems; for reviews of such developments see
\cite{2022arXiv220911371C,evensen2009data,evensen2022data}.
Despite widespread and growing adoption,
theory quantifying the accuracy of the ensemble Kalman filter, in relation to
the true filtering distribution, is limited.
Two exceptions are the important
contributions \cite{le2009large,mandel2011convergence} which concern accuracy in the large particle limit
in the setting where the underlying filtering problem is Gaussian. However, there is no proof that the
ensemble Kalman methodology can accurately approximate the desired filtering
distribution beyond the Gaussian setting. Indeed,
in general the methodology does not correctly reproduce the filtering
distribution; for examples and analysis in the setting of
filtering and Bayesian inverse problems see  \cite{apte2008bayesian,le2009large} and
\cite{MR3400030} respectively.
The aim of our work is to address this issue by proving accuracy of the
ensemble Kalman filter beyond the Gaussian setting; specifically our analysis applies
when the true filtering distribution is close to Gaussian, after
appropriate lifting to the joint space of state and data.
We perform the analysis for the mean field ensemble Kalman filter, focusing on
quantifying the effect of the Gaussian approximation underlying the
ensemble Kalman filter. We also study the Gaussian projected filter;
this filter can be seen as a mean field version of the
unscented Kalman filter. Both the ensemble and unscented Kalman mean
field models are defined in \cite{2022arXiv220911371C}.

The three primary contributions of our paper are: (i) to establish finite time stability estimates,
in an appropriate weighted TV metric which enables control of first and second order moments,
for various nonlinear maps required to define the nonlinear
Markov processes determining filter evolution; (ii) to use these results to establish
stability estimates controlling the error made by the mean field ensemble Kalman
filter in terms of the difference between the true filtering distribution and a nearby Gaussian;
and (iii) to deploy these estimates to prove closeness of the mean field ensemble Kalman filters
and the true filtering distribution, in the near-Gaussian setting. These results are also
established for the mean-field unscented Kalman filter. As well as making the three primary contributions,
the work suggests many questions for further analysis and the numerical analysis framework
we deploy (``consistency plus stability implies convergence'')
is a natural one in which to pursue these questions. In particular we assume bounded vector fields
and discuss only finite time error; overcoming these assumptions requires new ideas and is outside
the scope of this first paper.

Although our study of the accuracy of ensemble Kalman filters, beyond the Gaussian setting, is new,
there exists a growing body of literature analyzing ensemble Kalman methods from different perspectives.
In the context of long-time behaviour see the papers \cite{kelly2014well,tong2016nonlinear,MR3784489,de2018long} which focus primarily on the accuracy of
the filter in approximating the true trajectory over long time intervals; in contrast the papers \cite{gottwald2013mechanism,kelly2015concrete} demonstrate a mechanism for
filter divergence, an obstacle to obtaining accuracy.
Localization is widely used in practice and important to consider in overall understanding
of ensemble Kalman filter performance; see \cite{tong2018performance,tong2023localized}.
For analysis of filters in high dimensions see \cite{snyder2015performance,majda2018performance}.
For continuous-time and mean-field limits see \cite{lange2020mean,lange2021continuous,ertel2022analysis}.
In the context of inverse problems see~\cite{iglesias2013ensemble,schillings2017analysis,schillings2018convergence,huang2022iterated,huang2022efficient}.

\subsection{Overview of Paper}
\label{ssec:O}

In \cref{ssec:N} which follows we define useful notation employed throughout the
paper. Then, in \cref{ssec:FD} we define the filtering problem as employed throughout the
paper, building on this notation.
\Cref{sec:EKF} introduces
the mean field ensemble Kalman filter and proves our main approximation theorem, in the near Gaussian
setting; the theorem is based on a stability estimate which transfers distance between the true filter
and its Gaussian projection into the distance between the true filter and the ensemble Kalman
filter. In~\cref{sec:GPF} we define, and then state and prove related theory for, the Gaussian projected filter.
Although we provide theorems of a type not seen before for nonlinear Kalman filters,
and new methods of analysis to derive them,
our work leads to many substantial open problems; we conclude in \cref{sec:discussion} by highlighting those we have identified as being
of particular value to advancing the field. The reader may wish to study the
concluding \cref{sec:discussion}, in conjunction with our set-up in \cref{sec:EKF},
to understand the specific problem formulation for our main theorems and to appreciate
the substantial challenges to be met in order to build on, and go beyond, our theorems.
The key auxiliary results underpinning the proof of our main theorems are given in \cref{appendix:A,appendix:C};
these rely on technical results presented in \cref{app:A2}.

\subsection{Notation}
\label{ssec:N}

The Euclidean vector norm is denoted by~$\vecnorm{\placeholder}$,
and the corresponding operator norm on matrices is denoted by~$\matnorm{\placeholder}$.
For a symmetric positive definite matrix $\mat S \in \real^{n \times n}$,
the notation~$\vecnorm{\placeholder}_{\mat S}$ refers to the weighted norm $\vecnorm{ \mat S^{-1/2} \placeholder }$.
For a function $m \colon \real^n \to \real$ and $r \geq 0$,
we let $B_{L^{\infty}}(m,r)$ denote the $L^\infty$ ball of radius $r$,
centred at $m$;
and let $|\placeholder|_{C^{0,1}}$ denote the $C^{0,1}$ semi-norm, namely the Lipschitz constant.

We use symbol $\ind$ to denote independence of two random variables.
For $\vect m \in \real^n$ and $\mat \Sigma \in \real^{n \times n}$,
the notation $\normal(\vect m, \mat \Sigma)$ denotes the Gaussian distribution with mean~$\vect m$ and covariance~$\mat C$.
The notation~$\mathcal P(\real^n)$ denotes the set of probability measures over~$\real^n$,
and $\mathcal P^p(\real^n)$ is the set of probability measures over~$\real^n$ with finite moments up to order~$p$.
The notation $\mathcal P_{\rm c}(\real^n)$ is the set of probability measures over~$\real^n$ with continuous density with respect to the Lebesgue measure,
and notation $\mathcal G(\real^n)$ denotes the set of Gaussian probability measures over~$\real^n$.
We also introduce the Gaussian projection operator
$\op G\colon \mathcal P^2(\real^n) \to \mathcal G(\real^n)$ given by
\[
    \op G \mu = \normal\bigl(\mathcal M(\mu), \mathcal C(\mu)\bigr).
\]
We observe~\cite{MR2247587} that
\begin{equation*}
    \op G \mu = \argmin_{\nu \in \mathcal G} \kl{\mu}{\nu},
\end{equation*}
where $\kl{\mu}{\nu}$ is the Kullback--Leibler (KL) divergence of $\mu$ from $\nu$,
defined in~\eqref{eq:kl}.
Note that~$\op G$ defines a nonlinear mapping. We refer to $\op G$ as a projection because of
its characterization as finding closest point to $\mu$ with respect to the $\kl{\mu}{\placeholder}$ divergence.
Throughout this paper all probability measures have continuous Lebesgue density, because of our assumptions concerning the noise structure in the dynamics model and the data acquisition model. Thus
we abuse notation by using the same symbols for probability measures and their densities.
For a probability measure~$\mu \in \mathcal P(\real^n)$,
the notation $\mu(\vect u)$ for $\vect u \in \real^n$ refers to the Lebesgue density of~$\mu$ evaluated at~$\vect u$,
whereas $\mu[f]$ for a function~$f\colon \real^n \to \real$ is a  short-hand notation for $\int_{\real^n} f \d \mu$.

The notations~$\mathcal M(\mu)$ and~$\mathcal C(\mu)$ denote respectively the mean and covariance under $\mu$:
\[
    \mathcal M(\mu) = \mu[u], \qquad
    \mathcal C(\mu) = \mu \Bigl[\bigl(u - \mathcal M(\mu)\bigr) \otimes \bigl(u - \mathcal M(\mu)\bigr)\Bigr].
\]
The notation $\mathcal P_R(\real^n)$ for $R \geq 1$ refers to the subset of $\mathcal P(\real^n)$ of probability measures
whose first and second moments satisfy the bound
\begin{equation}
    \label{eq:space_local_lipschitz}
    \lvert \mathcal M(\mu) \rvert \leq R,
    \qquad
    \frac{1}{R^2} \mat I_n \preccurlyeq \mathcal C(\mu) \preccurlyeq R^2 \mat I_n.
\end{equation}
Here $\mat I_n$ denotes the identity matrix in $\real^{n \times n}$,
and~$\preccurlyeq$ is the partial ordering defined by the convex cone of positive semi-definite matrices.
Similarly, $\mathcal G_R(\real^n)$ is the subset of $\mathcal G(\real^n)$ of probability measures satisfying~\eqref{eq:space_local_lipschitz}.

For a probability measure $\pi \in \mathcal P(\real^d \times \real^K)$
associated with random variable $(u,y) \in \real^d \times \real^K$
we use the notation $\mathcal M^u(\pi) \in \real^d$, $\mathcal M^y(\pi) \in \real^K$
for the means of the marginal distributions,
and the notation $\mathcal C^{uu}(\pi) \in \real^{d \times d}$, $\mathcal C^{uy}(\pi) \in \real^{d \times K}$ and~$\mathcal C^{yy}(\pi) \in \real^{K\times K}$
for the blocks of the covariance matrix~$\mathcal C(\pi)$.
That is to say,
\[
    \mathcal M(\pi) =
    \begin{pmatrix}
        \mathcal M^u(\pi) \\
        \mathcal M^y(\pi)
    \end{pmatrix},
    \qquad
    \mathcal C(\pi) =
    \begin{pmatrix}
        \mathcal C^{uu}(\pi) & \mathcal C^{uy}(\pi) \\
        \mathcal C^{uy}(\pi)^\t & \mathcal C^{yy}(\pi)
    \end{pmatrix}.
\]

Throughout this work,
we employ the following weighted total variation distance.
\begin{definition}
    \label{definition:weighted_total_variation}
    Let $g\colon \real^n \to [1, \infty)$ be given by $g(v) = 1 + \vecnorm{v}^2$.
    We define the weighted total variation metric~$d_g \colon \mathcal P^2(\real^n) \times \mathcal P^2(\real^n) \to \real$ by
    \[
        d_g(\mu_1, \mu_2) = \sup_{\abs{f} \leq g} \bigl\lvert \mu_1[f] - \mu_2[f] \bigr\rvert,
    \]
    where the supremum is over all functions $f\colon \real^n \to \real$ which are bounded from above by~$g$
    pointwise and in absolute value.
\end{definition}
Metrics of this type have been employed previously in the literature.
See for example the early reference~\cite{MR1174380},
where a weighted total variation metric is used for stuyding ergodicity for Markov chains,
and the reference~\cite{MR4074698},
where a similar metric is employed in the context of perturbation of Markov chains.
See also~\cite{MR2857021},
where a direct proof of Harris' ergodic theorem relying on appropriate weighted total variation norms is given,
as well as~\cite[Theorem 6.15]{MR2459454},
which states that Wasserstein distances can be controlled by weighted total variation metrics.
There are other dual metrics on probability measures commonly employed in the literature,
such as the so-called maximum mean discrepancy,
which is based on a dual formulation in a reproducing kernel Hilbert space; see~e.g.~\cite{tolstikhin2016minimax}.
\begin{remark}
    \label{remark:total_variation}
    If $\mu_1, \mu_2$ have Lebesgue densities $\rho_1$, $\rho_2$,
    then the weighted metric in~\cref{definition:weighted_total_variation} satisfies
    \[
        d_g(\mu_1, \mu_2) = \int g(v) \abs*{\rho_1(v) - \rho_2(v)} \, \d v.
    \]
Unlike the usual total variation distance, this weighted total variation metric
enables control of the differences $\lvert \mathcal M(\mu_1) - \mathcal M(\mu_2) \rvert$ and $\lVert \mathcal C(\mu_1) - \mathcal C(\mu_2) \rVert$; several lemmas used to prove our main results rely on this control.
    More precisely, it is possible to show that
    for all $\mu_1, \mu_2 \in \mathcal P(\real^n)$ with finite second moments,
    it holds that
    \begin{align*}
        \bigl\lvert \mathcal M(\mu_1) - \mathcal M(\mu_2) \bigr\rvert &\leq \frac{1}{2} d_g(\mu_1, \mu_2), \\
        \bigl\lVert \mathcal C(\mu_1) - \mathcal C(\mu_2) \bigr\rVert &\leq \left(1 + \frac{1}{2} \vecnorm{\mathcal M(\mu_1) + \mathcal M(\mu_2)}\right) \, d_g(\mu_1, \mu_2).
    \end{align*}
    This is the content of~\cref{lemma:moment_bound},
    proved in the appendix.
\end{remark}

\subsection{Filtering Distribution}
\label{ssec:FD}

We consider a general setting in~\cref{sssec:GC},
and provide details for the Gaussian setting in particular in~\cref{sssec:GaC}.

\subsubsection{General Case}
\label{sssec:GC}

We consider the following stochastic dynamics and data model
\begin{subequations}
    \label{eq:filtering}
    \begin{alignat}{2}
        \label{eq:stochasic_dynamics}
        u_{j+1} &= \vect \Psi(u_j) + \xi_j, \qquad &&\xi_j \sim \normal(0, \mat \Sigma), \\
        \label{eq:data_model}
        y_{j+1} &= \vect h(u_{j+1}) + \eta_{j+1}, \qquad &&\eta_{j+1} \sim \normal(0, \mat \Gamma).
    \end{alignat}
\end{subequations}
Here $\{u_j\}_{j \in \llbracket 0, J \rrbracket}$ is the unknown state,
evolving~in $\real^d$,
and $\{y_j\}_{j \in \llbracket 0, J \rrbracket}$ are the observations, evolving in~$\real^K$.
We assume throughout the paper that the covariance matrices~$\mat \Sigma, \mat \Gamma$ are positive definite, that the function $h \colon \real^d \to \real^K$ is continuous, and
that the initial state is distributed according to a Gaussian distribution $u_0 \sim \normal(\vect m_0, \mat C_0)$
and that the following independence assumption is satisfied:
\begin{equation}
    \label{eq:independence}
    u_0 \ind \xi_0 \ind \dotsb \ind \xi_J \ind \eta_1 \ind \dotsb \ind \eta_{J+1}.
\end{equation}
The objective of probabilistic filtering is to sequentially estimate the
conditional distribution of the unknown state, given the data, as new data arrives.
The true filtering distribution $\mu_j$ is the conditional distribution of the state $u_j$ given
a realization $Y_j^\dagger = \{y_1^\dagger, \dotsc, y_j^\dagger \}$ of the data process up to step $j$.
Data $Y_j$ may be thought of as arising from a realization of \eqref{eq:filtering}; but the case of
model mis-specification, where $Y_j$ does not necessarily arise from \eqref{eq:filtering}, is also
of interest.

It is well-known~\cite{MR3363508,reich2015probabilistic} that the true filtering distribution evolves according to
\begin{equation}
    \label{eq:true_filtering}
    \mu_{j+1} = \op L_j \op P \mu_j.
\end{equation}
where $\op P$ and $\op L_j$ are maps on probability measures,
referred to respectively as the prediction and analysis steps in
the data assimilation community \cite{asch2016data}.
The operator $\op P\colon \mathcal P(\real^n) \to \mathcal P(\real^n)$ is linear and defined by the Markov kernel
associated with the stochastic dynamics~\eqref{eq:stochasic_dynamics}.
Its action on a probability measure with Lebesgue density~$\mu$ reads
\begin{equation}
    \label{eq:Markov_kernel}
    \op P \mu(u) = \frac{1}{\sqrt{(2\pi)^d \det \mat \Sigma}}\int_{\real^d} \exp \left( - \frac{1}{2} \vecnorm*{u - \vect \Psi(v)}_{\mat \Sigma}^2 \right) \mu(v)\, \d v.
\end{equation}
The operator $\op L_j\colon \mathcal P(\real^n) \to \mathcal P(\real^n)$ is a nonlinear map which formalizes the incorporation of the new datum~$y^\dagger_{j+1}$ using Bayes' theorem.
Its action on a probability measure in $\mathcal P(\real^d)$ with Lebesgue density~$\mu$ reads
\begin{equation}
    \label{eq:decomposition_of_the_analysis}
    \op L_j \mu(u) =
    \frac
    {\displaystyle \exp \left( - \frac{1}{2} \bigl\lvert y_{j+1}^{\dagger} - \vect h(u) \bigr\rvert_{\mat \Gamma}^2 \right)  \mu(u)}
    {\displaystyle \int_{\real^d} \exp \left( - \frac{1}{2} \bigl\lvert y_{j+1}^{\dagger} - \vect h(U) \bigr\rvert_{\mat \Gamma}^2 \right)  \mu(U) \, \d U}
\end{equation}
The operator $\op L_j$ effects a reweighting of the measure to which it applies,
with more weight assigned to the state values that are consistent with the observation.
It is convenient in this work to decompose the analysis map $\op L_j$
into the composition $\op B_j \op Q$.
The action of the operator $\op Q\colon \mathcal P(\real^d) \to \mathcal P(\real^d \times \real^K)$
on a probability measure in $\mathcal P(\real^d)$ with Lebesgue density~$\mu$ is given by
\begin{equation}
    \label{eq:definition_Q}
    \op Q \mu (u, y) = \frac{1}{\sqrt{(2\pi)^K \det \mat \Gamma}} \exp \left( - \frac{1}{2} \bigl\lvert y - \vect h(u) \bigr\rvert_{\mat \Gamma}^2 \right)  \mu(u).
\end{equation}
On the other hand,
the action of $\op B_j \colon \mathcal P_c(\real^d \times \real^K) \to \mathcal P(\real^d)$
on a probability measure with continuous Lebesgue density~$\mu$ is given by
\begin{equation}
    \label{eq:definition_B}
    \op B_j \mu(u) = \frac{\mu(u, y_{j+1}^{\dagger})}{\displaystyle \int_{\real^d} \mu(U, y_{j+1}^{\dagger}) \, \d U}.
\end{equation}
The operator $\op Q$ maps a probability measure with density $\mu$ into the
density associated with the joint random variable $\bigl(U,\vect h(U) + \eta\bigr)$,
where $U \sim \mu$ is independent of~$\eta \sim \normal(0, \mat \Gamma)$.
The operator~$\op B_j$ performs conditioning on the data~$y^\dagger_{j+1}$.
Map $\op Q$ is linear whilst $\op B_j$ is nonlinear.
We may thus write \eqref{eq:true_filtering} in the form
\begin{equation}
    \label{eq:true_filtering2}
    \mu_{j+1} = \op B_j \op Q \op P \mu_j.
\end{equation}
We note that for any $\mu\in \mathcal P(\real^d)$, the measure $\op Q\op P\mu\in \mathcal P_c(\real^d\times\real^K)$ has continuous Lebesgue density,
and so the iteration~\eqref{eq:true_filtering2} is well defined.

\subsubsection{Gaussian Case}
\label{sssec:GaC}

Now consider the setting in which $\Psi(\cdott)=M\cdott$, and $h(\cdott)=H\cdott$ for matrices
$\mat M \in \real^{d \times d}$ and~$\mat H \in \real^{K \times d}.$
In this case the filtering
distribution is Gaussian: $\mu_j=\normal(m_j,C_j).$ The Kalman filter~\cite{kalman1960new}
gives explicit evolution equations for the pair $(m_j,C_j) \in \real^d \times \real^{d \times d}.$
To write these down it is helpful to note that $\widehat{\mu}_{j+1}:={\op P}\mu_j$ is also
Gaussian: $\widehat{\mu}_{j+1}=\normal(\pmean_j,\pCov_j).$

Then $(\mu_{j+1},\hmu_{j+1})$ are determined from $\mu_j$ by the update formulae
\begin{subequations}
    \label{eq:update}
    \begin{align}
        \pmean_{j+1} &= \mat M \mean_{j},\\
        \pCov_{j+1} &= \mat M \Cov_{j}\mat M^\t + \mat \Sigma,\\
        \mean_{j+1} &= \pmean_{j+1} + \pCov_{j+1} \mat H^\t\left( \mat H\pCov_{j+1} \mat H^\t+ \mat \Gamma\right)^{-1} \left(\yd_{j+1} -  \mat H\pmean_{j+1}\right),\\
        \Cov_{j+1} &= \pCov_{j+1} - \pCov_{j+1} \mat H^\t\left( \mat H\pCov_{j+1} \mat H^\t+ \mat \Gamma\right)^{-1} \mat H \pCov_{j+1},
    \end{align}
\end{subequations}
where $\{\yd_{j}\}$ is the observation. Recall that this observation may be thought of as arising from a realization of \eqref{eq:filtering}; but, in the case of model mis-specification, may be derived from a different source.

\section{The Ensemble Kalman Filter}
\label{sec:EKF}

Following from expository discussion of the Gaussian case in ~\cref{ssec:MaF},
in~\cref{ssec:MF} we define the specific version of the mean-field ensemble Kalman filter
that we analyze here; other versions may be found in \cite{2022arXiv220911371C} and will be amenable
to similar analyses. In~\cref{ssec:SEKF} we prove our main stability theorem, showing that the
error between the true filter and its Gaussian projection may be used to control the error
between the true filter and the ensemble Kalman filter.
In~\cref{ssec:AEKF} we prove a corollary to this theorem, establishing that the mean field
ensemble Kalman filter accurately approximates the true filter for a specific class of
non-Gaussian problems.

\subsection{The Algorithm: Gaussian Case}
\label{ssec:MaF}

To motivate the mean-field ensemble Kalman filter we first consider the
Gaussian case and introduce the stochastic dynamical system
\begin{subequations}
    \label{eq:sd2nn_sub}
    \begin{align}
        \label{eq:sd2nn_sub_a}
        \hv_{j+1} &= Mv_j + \xi_j, \qquad & \xi_j \sim \normal(0, \mat \Sigma),\\
        \label{eq:sd2nn_sub_b}
        \hy_{j+1} &= H\hv_{j+1} + \eta_{j+1}, \qquad & \eta_{j+1} \sim \normal(0, \mat \Gamma)\\
        \label{eq:sd2nn_sub_c}
        v_{j+1} &= \hv_{j+1}+\pCov_{j+1}H^\top(H\pCov_{j+1}H^\top+ \mat \Gamma)^{-1}(\yd_{j+1}-\hy_{j+1}).
    \end{align}
\end{subequations}
Here $\pCov_{j+1}$ denotes the covariance of $\hv_{j+1}$.
A simple calculation reveals that, if~$v_0 \sim \normal(m_{0}, \mat C_{0})$,
then in fact $\hv_{j+1} \sim \normal(\pmean_{j+1},\pCov_{j+1})$ and $v_{j+1} \sim \normal(m_{j+1}, \mat C_{j+1})$, where the means and covariances are given by the Kalman filter \eqref{eq:update}.
Note that \eqref{eq:sd2nn_sub} constitutes a mean-field stochastic dynamical system because equation \eqref{eq:sd2nn_sub_c} requires knowledge of $\pCov_{j+1}$, which depends on the law~$\mu_j$ of~$v_j$.
The law of this mean-field stochastic dynamical system is thus equal to the law of the Kalman filter.
The analysis in \cite{le2009large,mandel2011convergence} is concerned with studying particle approximations of this Gaussian mean-field stochastic dynamical system, and convergence to the mean-field limit in the limit of infinite particles.
In contrast, our work concerns the study of mean-field stochastic dynamical systems,
but goes beyond the Gaussian setting. To this end, the next subsection defines the
mean-field ensemble Kalman filter in the general, non-Gaussian, setting.

\subsection{The Algorithm: General Case}
\label{ssec:MF}

The ensemble Kalman filter may be derived as a particle approximation of
various mean field limits \cite{2022arXiv220911371C}.
The specific mean field ensemble Kalman filter that
we study in this paper reads
\begin{subequations}
\label{eq:ensemble_kalman_mean_field}
\begin{alignat}{2}
    \widehat u_{j+1} &= \vect \Psi (u_j) + \xi_j, \qquad & \xi_j \sim \normal(0, \mat \Sigma), \\
    \widehat y_{j+1} &= \vect h(\widehat u_{j+1}) + \eta_{j+1},  \qquad & \eta_{j+1} \sim \normal(0, \mat \Gamma) \\
    \label{eq:mean_field_ekf_analysis}
    u_{j+1} &= \widehat u_{j+1} + \mathcal C^{uy}\left(\widehat \pi_{j+1}^K\right) \mathcal C^{yy}\left(\widehat \pi_{j+1}^K\right)^{-1} \bigl(y^\dagger_{j+1} - \widehat y_{j+1} \bigr),
\end{alignat}
\end{subequations}
where $\widehat \pi^K_{j+1} = {\rm Law}(\widehat u_{j+1}, \widehat y_{j+1})$
and independence of the noise terms~\eqref{eq:independence} is still assumed to hold.
The covariance matrices $\mat \Gamma, \mat \Sigma$ are still assumed to be positive definite,
so $\mathcal C^{yy}(\widehat \pi_{j+1}^K) \succ 0$
and the algorithm is well-defined.
See~\cref{ssec:N} for the definition of the covariance matrices that appear in~\eqref{eq:mean_field_ekf_analysis}.
\begin{remark}
    The mean-field EnKF algorithm may be derived as the best linear unbiased estimator (BLUE) of the predicted state, given the data; see \cite{2022arXiv220911371C}.
\end{remark}

Let us denote by~$\mu^K_j$ the law of $u_j$.
In order to rewrite the evolution of $\mu_j^K$ in terms of maps on probability measures,
let us introduce
\[
    \mathcal P^2_{\succ 0}\left(\real^d \times \real^K\right) := \Bigl\{ \pi \in \mathcal P^2\left(\real^d \times \real^K\right) : \mathcal C^{yy}(\pi) \succ 0 \Bigr\}.
\]
Then,
for a given $y^\dagger_{j+1}$,
we denote by $\op T_j \colon \mathcal P^2_{\succ 0}(\real^d \times \real^K) \to \mathcal P^2(\real^d)$ the map defined by
\begin{equation}\label{eq:Tmap}
\op T_j\pi =  \mathscr T(\placeholder, \placeholder; \pi, y^\dagger_{j+1})_{\sharp} \pi
\end{equation}
Here the subscript $_\sharp$ denotes the pushforward and,
for any given $\pi\in \mathcal P(\real^d \times \real^K)$ and $z\in \real^K$,
the map~$\mathscr T$ is affine in its first two arguments and given by
\begin{align}
    \nonumber
    \mathscr T(\placeholder, \placeholder; \pi, z) \colon
    &\real^{d} \times \real^K \to \real^d; \\
    \label{eq:mean_field_map}
    &(u, y) \mapsto u + \mathcal C^{uy}(\pi) \mathcal C^{yy}(\pi)^{-1} \bigl(z - y \bigr).
\end{align}
As demonstrated in~\cite{2022arXiv220911371C},
with this notation the evolution of the probability measure $\mu_j^K$ may be written in compact form as
\begin{equation}
    \label{eq:compact_mf_enkf}
    \mu_{j+1}^K = \op T_j\op Q \op P \mu_j^K.
\end{equation}
We now discuss the preceding map in relation to \eqref{eq:true_filtering2}.
The specific affine map $\mathscr T$ used in~\eqref{eq:mean_field_ekf_analysis} is determined by the measure $\widehat \pi^K_{j+1}$ (here equal to $\op Q \op P \mu_j^K$) and the data $y^\dagger_{j+1}$.
That the law of~$u_j$ in~\eqref{eq:ensemble_kalman_mean_field} evolves according to~\eqref{eq:compact_mf_enkf}
follows from the following observations:
\begin{itemize}
    \item
        If $u_{j} \sim \mu_j^K$, then $\widehat u_{j+1} \sim \op P \mu_j^K$,
        by definition of~$\op P$.

    \item
        Given the definition~\eqref{eq:definition_Q} of the operator~$\op Q$,
        the random vector~$(\widehat u_{j+1}, \widehat y_{j+1} )$ is distributed according to $\widehat \pi^K_{j+1}=\op Q \op P \mu_j^K$.

    \item
        Equation~\eqref{eq:mean_field_ekf_analysis} then implies that $u_{j+1} \sim \op T_j\op Q \op P \mu_j^K$.
\end{itemize}
As we show in~\cref{lemma:mean_field_map},
the operator~$\op T_j$ coincides with the conditioning operator~$\op B_j$ over the subset~$\mathcal G(\real^{d} \times \real^K) \subset \mathcal P(\real^d \times \real^K)$ of Gaussian probability measures.
Therefore, in the particular case where~$\mu_0$ is Gaussian,
which is a standing assumption in this paper,
and the operators~$\vect \Psi$ and~$\vect h$ are linear,
the mean field ensemble Kalman filter~\eqref{eq:compact_mf_enkf} reproduces the exact filtering distribution~\eqref{eq:true_filtering2},
which is then the Kalman filter itself; indeed this is what we show in~\cref{ssec:MaF}.
In the next section we provide error bounds between \eqref{eq:true_filtering2} and \eqref{eq:compact_mf_enkf} when
$\vect \Psi$ and~$\vect h$ are not assumed to be linear.

Our theorems  in~\cref{ssec:SEKF,ssec:AEKF} concern relationships between
the true filter \eqref{eq:true_filtering2} and the mean-field ensemble Kalman filter
\eqref{eq:compact_mf_enkf}. We study the setting in which $\vect \Psi$ and~$\vect h$
are not assumed to be linear, so that the true filter is not Gaussian, in~\cref{ssec:SEKF};
and then we study small perturbations away from the Gaussian setting that arise when the vector fields
$\vect \Psi$ and~$\vect h$  are close to constant in~\cref{ssec:AEKF}.
We recall that that $B_{L^{\infty}}(m,r)$ denotes the $L^\infty$ ball of radius $r$, centred at $m$.
The theorems hold under the following set of assumptions.

\begin{assumption}
    \label{assumption:ensemble_kalman}
    The following assumptions hold on the data $\{y_j^\dagger\}$, the vector fields $(\vect \Psi, \vect h)$
    and the covariances $(\mat \Sigma, \mat \Gamma).$

    \begin{assumpenum}[label=(H\arabic*)]
    \item
        \label{assumpenum:assumption1}
        Fix positive integer $J.$
        The data $Y^\dagger =\{y_j^\dagger\}_{j=1}^J$ lies in set $B_y \subset \real^{KJ}$ defined, for some $\kappa_y>0$, by
        \[
            B_y := \left\{Y^\dagger \in \real^{KJ}: \max_{j \in \range{0}{J}} \lvert y^\dagger_j \rvert \le \kappa_y\right\}.
        \]

    \item
        \label{assumpenum:assumption2}
        The function $\vect \Psi$ satisfies $\vect \Psi(\placeholder) \in B_{\Psi}:= B_{L^\infty}(0,\kappa_{\vect \Psi})$ for some $\kappa_{\vect \Psi}>0$.

    \item
        \label{assumpenum:assumption3}
        The function $\vect h$ is continuous and  satisfies $\vect h(\placeholder) \in B_{h}:= B_{L^\infty}(0,\kappa_{\vect h})$ for some $\kappa_{\vect h}>0$.
    \item
        \label{assumpenum:assumption4}
        The covariance matrices $\mat \Sigma$ and $\mat \Gamma$ appearing in~\eqref{eq:ensemble_kalman_mean_field} are positive definite:
        $\mat \Sigma \succcurlyeq \sigma \mat I_d$ and~$\mat \Gamma \succcurlyeq \gamma \mat I_K$ for positive $\sigma$ and $\gamma$.
    \end{assumpenum}
\end{assumption}

\subsection{Stability Theorem: Ensemble Kalman Filter}
\label{ssec:SEKF}

Roughly speaking,
our main result states that if the true filtering distributions $(\mu_j)_{j \in \range{1}{J}}$ are close to Gaussian
after appropriate lifting to the state/data space
then the distribution $\mu^K_j$ given by the mean field ensemble Kalman filter~\eqref{eq:compact_mf_enkf}
is close to the true filtering distribution~$\mu_j$  given by \eqref{eq:true_filtering2}
for all $j \in \llbracket 0, J \rrbracket$.
Recall that $|\placeholder|_{C^{0,1}}$ denotes the $C^{0,1}$ semi-norm, namely the Lipschitz constant.

\begin{theorem}
    [{\bf Stability for the Mean Field Ensemble Kalman Filter}]
    \label{theorem:main_theorem}
    Assume that the probability measures $(\mu_j)_{j \in \range{1}{J}}$ and  $(\mu^K_j)_{j \in \range{1}{J}}$ are obtained respectively from the dynamical systems~\eqref{eq:true_filtering2} and~\eqref{eq:compact_mf_enkf},
    initialized at the same Gaussian probability measure $\mu_0 = \mu_0^K\in~\mathcal G(\real^d)$.
    That~is,
    \[
        \mu_{j+1} = \op B_j \op Q \op P \mu_j, \qquad
        \mu^K_{j+1} = \op T_j \op Q \op P \mu_j^K.
    \]
    If \cref{assumption:ensemble_kalman} holds and $|\vect h|_{C^{0,1}} \leq \ell_{\vect h} < \infty$,
    then there exists~$C = C(J,\kappa_y, \kappa_{\vect \Psi}, \kappa_{\vect h}, \ell_{\vect h}, \mat \Sigma, \mat \Gamma)$
    such that
    \[
        d_g(\mu^K_J, \mu_J) \leq C  \max_{j \in \range{0}{J-1}} d_g( \op Q \op P \mu_j, \op G \op Q \op P \mu_j).
    \]
\end{theorem}

\begin{remark}
The constant $C$ in the stability estimate
is uniform across realizations of the
data from $B_y.$ Indeed, since $h$ is assumed bounded, the bound $\kappa_y$ on the data will hold with
probability exponentially close to $1$,\footnote{With respect to the noise realization leading
to the data.} for $\kappa_y \gg 1$, when there is no model mis-specification.
Relaxing the assumptions of
bounded $\Psi$ and $h$, for example to the setting of linear plus bounded functions, is technically
challenging and will be left for future work. Likewise relaxing positivity assumptions on the
noise leads to substantial technical obstacles, deferred for future work. And finally, relaxing the
assumption that $J$ is finite will require further structural assumptions on the long-time
behaviour of the filter, and is also deferred for future work.
\end{remark}

\begin{remark}
The stability theorem shows that the error made by the ensemble Kalman filter is controlled by the
difference between the true filter, lifted to the joint space of state and observations,
and its Gaussian projection. To interpret the result it is thus important to have intuition
about what it means for a measure to be close to its Gaussian projection in the $d_g$ metric.
To this end note that, by \cref{remark:total_variation},
it is necessary
that the means and covariances of the measure and its Gaussian projection are close, for the
measures to be close in the $d_g$ metric.  This is automatically satisfied for the difference
between a measure and its Gaussian projection, by construction. Since closeness in $d_g$
requires the expectations of all functions growing no faster than quadratic to be close,
a useful rule of thumb for practitioners is that the quantity $d_g( \op Q \op P \mu_j, \op G \op Q \op P \mu_j)$ is small when matching first and second moments allows control of all quadratically bounded
functions. This will happen, for example, when $\Psi$ and $h$ are small nonlinear perturbations of
affine functions; and hence it will happen when $\Psi$ and $h$ are small nonlinear perturbations of
either linear functions or constant functions. For technical reasons
(see previous remark) this paper excludes the case of linearly growing functions, but we do study small perturbations of constant functions in Corollary \ref{theorem:main_theorem2}. There is also reason to expect the filtering distribution to be close to Gaussian when the data volume is high, or noise is small, and some version of observability applies; then central limit theorems for inverse problems, such as the Bernstein von Mises theorem, may in the future be developed to quantify this assertion \cite{van2000asymptotic}; for work pointing in this direction, but in the infinite dimensional case where Bernstein von Mises type results
are harder to establish, see~\cite{nickl2023posterior}.
\end{remark}

The proof presented hereafter relies on a number of auxiliary results,
which are summarized below and proved in~\cref{appendix:A}.
\begin{enumerate}
    \item
        For any probability measure~$\mu$,
        the first moments of the probability measures~$\op P \mu$ and~$\op Q \op P \mu$ are bounded from above,
        and their second moments are bounded both from above and from below.
        The constants in these bounds depend only on the parameters $\kappa_{\vect \Psi}$, $\kappa_{\vect h}$, $\mat \Sigma$ and~$\mat \Gamma$.
        See~\cref{lemma:bound_on_Pmu,lemma:bound_on_QPmu}.

    \item
        For any Gaussian measure $\mu \in \mathcal G(\real^{d} \times \real^{K})$,
        it holds that~$\op B_j \mu = \op T_j \mu$.
        See~\cref{lemma:mean_field_map} and also~\cite{2022arXiv220911371C}.

    \item
        The map~$\op P$ is globally Lipschitz on $\mathcal P(\real^d)$ for the metric $d_g$,
        with a Lipschitz constant~$L_{\op P}$ depending only on the parameters $\kappa_{\vect \Psi}$ and $\mat\Sigma$.
        See~\cref{lemma:lipschitz_p}.

    \item
        The map~$\op Q$ is globally Lipschitz on $\mathcal P(\real^d)$ for the metric $d_g$,
        with a Lipschitz constant~$L_{\op Q}$ depending only on the parameters $\kappa_{\vect h}$ and $\mat \Gamma$.
        See~\cref{lemma:lipschitz_Q}.

    \item
        The map $\op B_j$ satisfies for any $\pi \in \ran(\op Q \op P)\subset \mathcal P(\real^d\times\real^K)$  the bound
        \[
            \forall j \in \llbracket 0, J \rrbracket,
            \qquad
            d_{g}(\op B_j \op G \pi, \op B_j \pi)
            \leq C_{\op B} d_{g}(\op G \pi, \pi),
        \]
        where $C_{\op B} = C_{\op B}(\kappa_y, \kappa_{\vect \Psi}, \kappa_{\vect h}, \mat \Sigma, \mat \Gamma)$.
        This statement concerns the stability of the $\op B_j$ operator between a measure and its Gaussian approximation.
        See~\cref{lemma:missing_lemma}.

    \item
        The map $\op T_j$ satisfies the following bound:
        for all $R \ge 1$, it holds for all probability measures~$\pi \in \mathcal P_R(\real^d \times \real^K)$ and $p \in \ran(\op Q \op P)\subset \mathcal P(\real^d\times\real^K)$ that
        \[
            \forall j \in \llbracket 0, J \rrbracket,
            \qquad
            d_g(\op T_j \pi, \op T_j p)
            \leq L_{\op T} \, d_g(\pi, p),
        \]
        for a constant $L_{\op T} = L_{\op T}(R, \kappa_y, \kappa_{\vect \Psi}, \kappa_{\vect h}, \mat \Sigma, \mat \Gamma)$.
        This statement may be viewed as a local Lipschitz continuity result in the case where the second argument of $d_g$ is restricted to the range of $\op Q \op P$.
        See~\cref{lemma:lipschitz_mean_field_affine}.
\end{enumerate}

\begin{proof}
    [Proof of \cref{theorem:main_theorem}]
    In what follows we refer to the preceding itemized list to clarify the proof.
    For notational simplicity it is helpful to define the following measure of the difference
    between the true filtering distribution and its Gaussian projection:
    \begin{equation}
        \label{eq:epsilon}
        \varepsilon:=  \max_{j \in \range{0}{J-1}} d_g( \op Q \op P \mu_j, \op G \op Q \op P \mu_j).
    \end{equation}
    Assume throughout the following that $j \in \range{0}{J-1}.$
    The main idea of the proof results from the following use of the triangle inequality
    \begin{subequations}
    \begin{align}
        &d_g(\mu^K_{j+1}, \mu_{j+1})
        = d_g\left(\op T_j \op Q \op P \mu_j^K, \op B_j \op Q \op P \mu_j\right) \\
        &\quad \leq d_g\left(\op T_j \op Q \op P \mu_j^K, \op T_j \op Q \op P \mu_j\right) +  d_g\left(\op T_j \op Q \op P \mu_j, \op T_j \op G \op Q \op P \mu_j\right)
        +  d_g\left(\op B_j \op G \op Q \op P \mu_j, \op B_j \op Q \op P \mu_j\right).\label{eq:main_idea_main_proofB}
    \end{align}
    \end{subequations}
    We have used the fact that the second argument of the second term on the right-hand side indeed coincides with the first argument of the third term because $\op T_j \op G \op Q \op P \mu_j = \op B_j \op G \op Q \op P \mu_j$ by~item 2 (\cref{lemma:mean_field_map}.)
    By~item~1 (\cref{lemma:bound_on_QPmu}),
    there is a constant~$R \ge 1$, depending on the covariance matrices~$\mat \Sigma$, $\mat \Gamma$ and the bounds~$\kappa_{\vect \Psi}$ and~$\kappa_{\vect h}$
    from \cref{assumption:ensemble_kalman}, such that $\ran(\op Q \op P) \subset \mathcal P_R(\real^d \times \real^K)$.
    By~items 3, 4 and 6 (\cref{lemma:lipschitz_mean_field_affine,lemma:lipschitz_Q,lemma:lipschitz_p}), the first term in~\eqref{eq:main_idea_main_proofB} satisfies
    \begin{align}
        \label{eq:main_theorem_1}
        d_g\left(\op T_j \op Q \op P \mu_j^K, \op T_j \op Q \op P \mu_j \right)
        &\leq L_{\op T}(R) L_{\op Q} L_{\op P} d_g\left(\mu_j^K, \mu_j \right),
    \end{align}
    where, for conciseness, we omitted the dependence of the constants on~$\kappa_y, \kappa_{\vect \Psi}, \kappa_{\vect h}, \mat \Sigma, \mat \Gamma$.
    Equation~\eqref{eq:main_theorem_1} establishes that the composition of maps $\op T_j \op Q \op P$ is globally Lipschitz over~$\mathcal P(\real^d)$.
    Since~$\op G \op Q \op P \mu_j\in \mathcal G_R(\real^d \times \real^K)$ by definition of $R$,
    the second term in~\eqref{eq:main_idea_main_proofB} may be bounded using item 6 (\cref{lemma:lipschitz_mean_field_affine}) and the definition in~\eqref{eq:epsilon} of $\varepsilon$:
    \begin{equation*}
        d_g\left(\op T_j \op Q \op P \mu_j, \op T_j \op G \op Q \op P \mu_j\right)
        \leq L_{\op T}(R) d_g\left(\op Q \op P \mu_j, \op G \op Q \op P \mu_j\right)
        \le L_{\op T}(R) \varepsilon.
    \end{equation*}
    Finally, the third term in~\eqref{eq:main_idea_main_proofB} can be bounded using item 5 (\cref{lemma:missing_lemma}) and the definition in~\eqref{eq:epsilon} of $\varepsilon$:
    \[
        d_g\left(\op B_j \op G \op Q \op P \mu_j, \op B_j \op Q \op P \mu_j\right)
        \leq C_{\op B} d_g\left(\op G \op Q \op P \mu_j, \op Q \op P \mu_j\right) \le  C_{\op B} \varepsilon.
    \]
    Therefore, letting~$\ell = L_{\op T}(R) L_{\op Q} L_{\op P}$, we have shown that
    \begin{align*}
        d_g(\mu^K_{j+1}, \mu_{j+1})
        &\leq \ell d_g(\mu^K_{j}, \mu_{j}) + \bigl(L_{\op T}(R) + C_{\op B}\bigr) \varepsilon
    \end{align*}
    and the conclusion follows from the discrete Gronwall lemma, since $\mu_0 = \mu_0^K.$
\end{proof}

\subsection{Approximation Theorem: Ensemble Kalman Filter}
\label{ssec:AEKF}

\cref{theorem:main_theorem} shows that the ensemble Kalman filter error can be made arbitarily small
if the true filter is arbitarily close to its Gaussian projection, in state-observation space.
This 'closeness to Gaussian' assumption can be satisfied in our setting of bounded vector fields
by considering small perturbations of constant vector fields and is the content of the following corollary.
The result provides a first step in the analysis of the accuracy of the ensemble Kalman filter,
as an approximation of the true filter;
desirable generalizations of what we prove here are
discussed in the conclusions, in~\cref{sec:discussion}.

\begin{corollary}
    [{\bf Accuracy for the Mean Field Ensemble Kalman Filter}]
    \label{theorem:main_theorem2}
    Let $(\mat \Sigma, \mat \Gamma)$ satisfy~\cref{assumpenum:assumption4}.
    Suppose that $\Psi_0\colon \real^d \to \real^d$ and $\vect h_0 \colon \real^d \to \real^K$ are functions taking constant values and denote by $B_{\Psi_0,h_0}(r)$ the set of all functions $(\Psi,h)$ satisfying
    $\Psi \in B_{L^{\infty}}(\Psi_0,r)$, $h \in B_{L^{\infty}}(h_0,r)$ and~\cref{assumpenum:assumption2,assumpenum:assumption3}.
    Assume also that $|h|_{C^{0,1}} \leq \ell_{\vect h} < \infty$ and
    denote by $(\mu_j)_{j \in \range{1}{J}}$ and  $(\mu^K_j)_{j \in \range{1}{J}}$ the probability measures obtained respectively from the dynamical systems~\eqref{eq:true_filtering2} and~\eqref{eq:compact_mf_enkf},
    initialized at the same Gaussian probability measure $\mu_0 = \mu_0^K\in~\mathcal G(\real^d)$.
    That~is,
    \[
        \mu_{j+1} = \op B_j \op Q \op P \mu_j, \qquad
        \mu^K_{j+1} = \op T_j \op Q \op P \mu_j^K.
    \]
    Then for any $\epsilon>0$ there exists $\delta>0$ such that
    \[
        \sup_{Y^\dagger \in B_{y}} \sup_{(\Psi,h) \in B_{\Psi_0,h_0}(\delta)} d_g(\mu^K_J, \mu_J) \leq \epsilon.
    \]
\end{corollary}

\begin{proof}
The result follows from~\cref{theorem:main_theorem} and~\cref{prop:1} in~\cref{appendix:C}.
\end{proof}

\section{Generalization: Gaussian Projected Filter}
\label{sec:GPF}

We generalize the main result to the Gaussian projected filter defined in~\cite{2022arXiv220911371C}.
This algorithm may be viewed as a mean field version of the unscented Kalman filter \cite{julier1997new}.
Using a similar approach to that adopted in the previous section,
we prove a similar stability theorem and accuracy corollary.
The Gaussian projected filter is defined by the iteration
\begin{equation}
    \label{eq:gaussian_projection_filter}
    \mu^G_{j+1} = \op B_j \op G \op Q \op P \mu^G_j.
\end{equation}
This iteration is obtained from \eqref{eq:true_filtering2} by inserting a projection
onto Gaussians  before the conditioning step $\op B_j.$ Because conditioning of Gaussians
on linear observations preserves Gaussianity, the preceding map generates sequence
of measures $\{\mu^G_j\}$ in $\mathcal G(\real^d).$

\begin{remark}
    As shown in~\cref{lem:newform},
    the evolution~\eqref{eq:gaussian_projection_filter} may also be rewritten
    in the form
    \begin{equation}
        \label{eq:gaussian_projection_filter2}
        \mu^G_{j+1} = \op G \op T_j \op Q \op P \mu^G_j.
    \end{equation}
    This shows that the Gaussian projected filter may also be obtained from the mean field ensemble Kalman filter~\eqref{eq:compact_mf_enkf} by adding a Gaussian projection step after
    the conditioning.
\end{remark}

Like~\eqref{eq:compact_mf_enkf}
the filter~\eqref{eq:gaussian_projection_filter} reproduces the true filtering distributions when~$\vect \Psi$ and $\vect h$ are linear,
assuming that~$\mu_0$ is Gaussian.
The following theorem quantifies closeness
of the Gaussian projected filter to the true filtering distribution when the
linearity assumption on $\vect \Psi$ and $\vect h$ is relaxed.

\begin{theorem}
    [{\bf Stability for the Gaussian Projected Filter}]
    \label{theorem:main_theorem_gaussian_projection}
    Assume that the probability measures $(\mu_j)_{j \in \range{1}{J}}$ and  $(\mu^K_j)_{j \in \range{1}{J}}$ are obtained respectively from the dynamical systems~\eqref{eq:true_filtering2} and~\eqref{eq:gaussian_projection_filter},
    initialized at the same Gaussian probability measure $\mu_0 = \mu_0^G$.
    That is,
    \[
        \mu_{j+1} = \op B_j \op Q \op P \mu_j, \qquad
        \mu^G_{j+1} = \op B_j \op G \op Q \op P \mu_j^G.
    \]
    If \cref{assumption:ensemble_kalman} holds, then there exists~$C = C(J,\kappa_y, \kappa_{\vect \Psi}, \kappa_{\vect h}, \mat \Sigma, \mat \Gamma)$
    such that
    \[
        d_g(\mu^G_J, \mu_J) \leq C  \max_{j \in \range{0}{J-1}} d_g( \op Q \op P \mu_j, \op G \op Q \op P \mu_j).
    \]
\end{theorem}

The proof of this error estimate relies on the auxiliary results~\cref{lemma:bound_on_QPmu,lemma:lipschitz_p,lemma:lipschitz_Q,lemma:missing_lemma} already discussed,
as well as the following two additional results.
\begin{itemize}
    \item
        The map $\op G$ is locally Lipschitz for the metric $d_g$,
        in the sense that for any~$R \ge 1$,
        this map is Lipschitz continuous over the set $\mathcal P_R(\real^d)$ given in~\eqref{eq:space_local_lipschitz}.
        The associated Lipschitz constant is denoted by $L_{\op G} = L_{\op G}(R)$ and diverges as $R \to \infty$.
        This result is proved in~\cref{lemma:lipschitz_G},
        which relies on an auxiliary result shown in~\cref{lemma:dg_gaussians} on the distance between Gaussians in the $d_g$ metric.

    \item
        The map $\op B_j$ is locally Lipschitz for the metric $d_g$ over Gaussians,
        in the sense that for any~$R \ge 1$ and any $j \in \range{0}{J-1}$,
        this map is Lipschitz continuous over $\mathcal G_R(\real^d \times \real^K)$.
        The associated Lipschitz constant is denoted by $L_{\op B} = L_{\op B}(R, \kappa_y)$.
        See~\cref{lemma:lipschitz_l2}.
\end{itemize}
\begin{proof}
    We obtain by the triangle inequality that
       \begin{align*}
        d_{g}(\mu_{j+1}^G, \mu_{j+1})
        &= d_{g}(\op B_j \op G \op Q \op P \mu_{j}^G, \op B_j \op Q \op P \mu_{j})\\
        &\leq d_{g}(\op B_j \op G \op Q \op P \mu_{j}^G, \op B_j \op G \op Q \op P \mu_{j}) + d_{g}(\op B_j \op G \op Q \op P \mu_{j}, \op B_j \op Q \op P \mu_{j}).
    \end{align*}
    It follows from~\cref{lemma:bound_on_QPmu,lemma:lipschitz_p,lemma:lipschitz_Q,lemma:lipschitz_G,lemma:lipschitz_l2} that the composition of maps~$\op B_j \op G \op Q \op P$ is globally Lipschitz continuous on $\mathcal P(\real^d)$ with a constant
    $\ell = \ell(R,\kappa_y, \kappa_{\vect \Psi}, \kappa_{\vect h}, \mat \Sigma, \mat \Gamma)$,
    where $R = R(\kappa_{\vect \Psi}, \kappa_{\vect h}, \mat \Sigma, \mat \Gamma)$ is a positive constant such that $\ran(\op Q \op P) \subset \mathcal P_R(\real^d \times \real^K)$.
    Therefore, the first term on the right hand side may be bounded by
    \begin{align*}
        d_{g}(\op B_j \op G \op Q \op P \mu_{j}^G, \op B_j \op Q \op P \mu_{j})
        \leq \ell d_{g}(\mu_{j}^G, \mu_{j}),
    \end{align*}
    For notational simplicity we again let~$\varepsilon$ be as in~\eqref{eq:epsilon}.
    Using~\cref{lemma:missing_lemma} and definition of~$\varepsilon$,
    the second term may be bounded from above by
    \begin{align*}
        d_{g}(\op B_j \op G \op Q \op P \mu_{j}, \op B_j \op Q \op P \mu_{j})
        &\leq C_{\op B}\, d_{g}(\op G \op Q \op P \mu_{j}, \op Q \op P \mu_{j})
        \leq C_{\op B} \, \varepsilon,
    \end{align*}
    where $C_{\op B} = C_{\op B}(\kappa_y, \kappa_{\vect \Psi}, \kappa_{\vect h}, \mat \Sigma, \mat \Gamma)$ is the constant from~\cref{lemma:missing_lemma}.
    The proof can then be concluded in the same way as that of~\cref{theorem:main_theorem}.
\end{proof}

\begin{corollary}
    [{\bf Accuracy for the Gaussian Projected Filter}]
    \label{theorem:main_theorem_gaussian_projection2}
    Let $(\mat \Sigma, \mat \Gamma)$ satisfy~\cref{assumpenum:assumption4}.
    Suppose that $\Psi_0\colon \real^d \to \real^d$ and $\vect h_0 \colon \real^d \to \real^K$ are functions taking constant values and denote by $B_{\Psi_0,h_0}(r)$ the set of all functions $(\Psi,h)$ satisfying
    $\Psi \in B_{L^{\infty}}(\Psi_0,r)$, $h \in B_{L^{\infty}}(h_0,r)$ and~\cref{assumpenum:assumption2,assumpenum:assumption3}.
    Assume that the probability measures $(\mu_j)_{j \in \range{1}{J}}$ and  $(\mu^K_j)_{j \in \range{1}{J}}$ are obtained respectively from the dynamical systems~\eqref{eq:true_filtering2} and~\eqref{eq:gaussian_projection_filter},
    initialized at the same Gaussian probability measure $\mu_0 = \mu_0^K\in~\mathcal G(\real^d)$.
    That~is,
    \[
        \mu_{j+1} = \op B_j \op Q \op P \mu_j, \qquad
        \mu^G_{j+1} =\op B_j \op G \op Q \op P \mu_j^G.
    \]
    Then for any $\epsilon>0$ there exists $\delta>0$ such that
    \[
        \sup_{Y^\dagger \in B_{y}} \sup_{(\Psi,h) \in B_{\Psi_0,h_0}(\delta)} d_g(\mu^G_J, \mu_J) \leq \epsilon.
    \]
\end{corollary}

\section{Discussion and Future Directions}
\label{sec:discussion}
We have provided the first analysis of the error incurred by ensemble Kalman filters,
as approximations of true filtering distribution, beyond the linear Gaussian setting. We
have employed an appropriate weighted TV metric and obtained new stability
estimates in this metric, in order to establish the approximation results.
Our framing of the problem is motivated
by the framing of the analysis of the particle filter contained in
\cite[Section 1]{MR3375889}, a systematization of the original proof of convergence of
the particle filter appearing in \cite{del1997nonlinear}. Although it introduces new methodology
and theoretical results for nonlinear Kalman filters,
our work leaves open numerous avenues for further
analysis; we now highlight those that we identify as particularly important.
These remaining open problems are substantial, but the framework we map out in this paper is an appropriate one in which to address them.

\begin{enumerate}
    \item[(i)] Our theorems concern the mean field limit
        of the ensemble Kalman filter; it is of interest to study finite
        particle approximations of the mean field limit, along the lines of the work
        in \cite{le2009large,MR2860672} and analogous continuous time analyses in \cite{ding2021ensemble,ding2021ensembleb,ding2020ensemble}. Analysis of
        mean field limits of interacting particle systems is an established field;
        interfacing the natural metrics employed for such analyses (Wasserstein)
        with those employed here (weighted TV) will be required.

    \item[(ii)] We have made boundedness assumptions on $\Psi(\placeholder)$ and $h(\placeholder)$ in this paper. Developing proofs which relax these assumptions, allowing consideration of small nonlinear
        perturbations of the Kalman filter setting, for example, will be very valuable.

    \item[(iii)] Our error bounds are for a finite number of steps and,
        as is typical of finite time error estimates that employ a consistency plus stability implies convergence approach, lead to error constants
        which grow exponentially with the time horizon. The literature on analysis of numerical
        methods for non-autonomous dynamical systems demonstrates that going beyond
        finite time error estimates that
        grow in time is, in general, not possible \cite{stuart1998dynamical}; in that context, generalizing to
        long-time estimates requires assumptions on the long-term stability, or even
        ergodicity, of the dynamical system. Such long-term stability issues are
        widely studied for the true filtering distribution -- see~\cite{ocone1996asymptotic,del2001stability,van2009uniform,tong2012ergodicity,crisan2020stable,sanz2015long} for example; they are complicated by the fact that the nonlinear evolution equation for
        the filtering distribution is non-autonomous due to the observation signal.
        In the context of the EnKF such stability estimates are used
        in the paper \cite{MR3784489} which identifies linear and Gaussian filtering problems in which it is possible to generalize the large particle asymptotic
        analyses of~\cite{le2009large,mandel2011convergence}. There is also analysis
        of the EnKF for nonlinear non-Gaussian problems, such as the data
        assimilation problem for the Navier-Stokes equation, but this work concerns
        only accuracy of state estimation, not the entire filtering distribution~\cite{kelly2014well,biswas2024unified}.

    \item[(iv)] We have assumed a model for evolution of the state which
        employs additive Gaussian noise. Generalizing to a general Markov
        chain would be valuable.  Relatedly it is of interest to relax the assumption of including noise in the dynamical system for the state, to allow for deterministic
        dynamics.

    \item[(v)] We have studied a (widely used) version of the ensemble Kalman
        filter which employs a specific transport map to approximate the conditioning
        step in the filter. Other transport maps are also used in practice,
        such as that leading to the ensemble square root filter; as outlined in \cite[Section 2]{2022arXiv220911371C}.
        Analyzing these other methods would be of great interest.

    \item[(vi)] Filtering can be used to solve inverse problems, as outlined in
        \cite[Section 4]{2022arXiv220911371C}; in particular some of these
        methods rely on filtering over the infinite time horizon, the analysis of which will require new ideas such as in \cite{durmus2020elementary}.

    \item[(vii)] Continuous time versions of our analysis would be of interest,
        and their relationship to the Kushner--Stratonovich equation \cite{cricsan1999interacting};
        the paper \cite{MR3784489}, which studies mean field limits of the ensemble
        Kalman--Bucy filter, may be important in this context; furthermore study of
        the ensemble Kalman sampler \cite{garbuno2020interacting,garbuno2020affine,carrillo2019wasserstein}
        for inverse problems, would be of interest.

    \item[(viii)] The papers \cite{le2009large} and \cite{ding2020ensemble} propose reweighting
        of ensemble Kalman methods to make them unbiased; further analysis of this idea, and the
        development of new methods that can carry this out in a derivative-free fashion, would be of
        interest.

    \item[(ix)] Ensemble Kalman filters in practice employ very small numbers
        of ensemble members, and use covariance (spatial) localization when the unknown
        states are fields; developing analyses which account for low rank approximations
        and localization would be a valuable step for the field.
        Furthermore it would also be of interest to study the problem of quantization of measures \cite{graf2000foundations} in the context of filtering; however this is a very hard problem even in simple low dimensional settings \cite{caglioti2016quantization} and studying it for the evolution of the
        filtering distribution will require substantial new ideas.

    \item[(x)] Particle filters suffer from a curse of dimensionality on
        certain families of high-dimensional problems~\cite{MR2459233,snyder2008obstacles}. Studying ensemble Kalman methods, to determine whether they ameliorate this
        issue in the near-Gaussian setting, would be of interest.
        There are a variety of different models that can be employed to characterize families of high dimensional filtering problems \cite{2018arXiv181006191S}
        and such structural assumptions will undoubtedly affect
        the results that might be obtained for ensemble Kalman methods.

    \item[(xi)]
        Exploiting small noise or large data limits, to establish that the error \eqref{eq:epsilon} is small,
        and then using this fact to analyze the error in the ensemble Kalman filter, would be of great interest.
\end{enumerate}

\vspace{0.1in}
\paragraph{Acknowledgments}
We are grateful to the reviewers for very useful comments on a previous version of this paper.
In particular, we thank the reviewers for pointing out the elementary inequality in \cref{lemma:elementary_gaussians},
and for showing us how this inequality could be used to significantly improve the proofs of~\cref{lemma:gaussians_relation_1_and_inf_norms,lemma:lipschitz_density,lemma:new_lemma_revision}.
Notably, the constructive proof of~\cref{lemma:gaussians_relation_1_and_inf_norms} was proposed by a reviewer.
JAC was supported by the Advanced Grant Nonlocal-CPD (Nonlocal PDEs for Complex Particle Dynamics: Phase Transitions, Patterns and Synchronization) of the European Research Council Executive Agency (ERC) under the European Union’s Horizon 2020 research and innovation programme (grant agreement No. 883363).
JAC was also partially supported by the Engineering and Physical Sciences Research Council (EPSRC) under grants EP/T022132/1 and EP/V051121/1.
FH is supported by start-up funds at the California Institute of Technology.
FH was also supported by the Deutsche Forschungsgemeinschaft (DFG, German Research
Foundation) via project 390685813 - GZ 2047/1 - HCM.
The work of AMS is supported by a Department of Defense Vannevar Bush Faculty Fellowship,
and by the SciAI Center, funded by the Office of Naval Research (ONR), under Grant Number N00014-23-1-2729.
UV is partially supported by the European Research Council (ERC) under the European Union's Horizon 2020 research and innovation programme (grant agreement No 810367),
and by the Agence Nationale de la Recherche under grants ANR-21-CE40-0006 (SINEQ) and ANR-23-CE40-0027 (IPSO).

\vspace{0.1in}
\printbibliography

\appendix


\section{Auxiliary Results}
\label{app:A2}

We begin by presenting an elementary result used throughout the article.
\begin{lemma}
    \label{lemma:auxiliary_ineq_second_moment}
    Suppose that $X$ is a random variable with values in $\real^d$ and finite second moment,
    and let $\vect m := \expect [X]$.
    Then
    \begin{equation}
        \label{eq:second_moment_vector}
        \expect \Bigl[ (X - \vect m)(X - \vect m)^\t\Bigr] = \expect \bigl[XX^\t\bigr] - \vect m \vect m^\t
    \end{equation}
    and
    \begin{equation}
        \label{eq:inequality_second_moment}
        \forall \vect a \in \real^d,
        \qquad
        \expect \Bigl[ (X - \vect a)(X - \vect a)^\t\Bigr]
        \succcurlyeq \expect \Bigl[ (X - \vect m)(X - \vect m)^\t\Bigr].
    \end{equation}
\end{lemma}
\begin{proof}
    We have
    \begin{align*}
        \expect \Bigl[ (X - \vect a)(X - \vect a)^\t\Bigr]
        &= \expect \Bigl[ \bigl((X - \vect m) + (\vect m - \vect a)\bigr)\bigl((X - \vect m) + (\vect m - \vect a)\bigr)^\t\Bigr] \\
        &= \expect \Bigl[ (X - \vect m)(X - \vect m)^\t\Bigr] + (\vect m - \vect a)(\vect m - \vect a)^\t.
    \end{align*}
    Taking $\vect a = \vect 0$, we obtain~\eqref{eq:second_moment_vector}.
    In addition,
    since the second term in the last expression is positive semidefinite,
    the inequality~\eqref{eq:inequality_second_moment} follows.
\end{proof}

The following lemma is very similar to~\cite[Lemma 3.1]{MR2730330};
the only difference is that the weight on the left-hand side is given by $1 + \abs{u}^2$ instead of $u^2$.
We give the proof for completeness.

\begin{lemma}
    [Generalized Pinsker inequality]
    \label{sub:generalized_pinsker}
    Let $g(u) = 1 + \vecnorm{u}^2$ and assume that $\mu_1$, $\mu_2$ are probability measures over $\real^d$ satisfying $\mu_1[g^2] < \infty$ and $\mu_2[g^2] < \infty$.
    Then, if $\mu_1 \ll \mu_2$,
    \[
        d_g(\mu_1, \mu_2)^2 \leq 2 \bigl( \mu_1\left[g^2\right] + \mu_2\left[g^2\right] \bigr) \kl{\mu_1}{\mu_2}.
    \]
\end{lemma}
\begin{proof}
    Denote the density by $f := \frac{\d \mu_1}{\d \mu_2}$, noting that this is non-negative.
    Applying Taylor's formula to the function $\ell\colon u \mapsto u \log u$ around $u = 1$,
    we deduce that
    \[
        \forall u \geq 0, \qquad
        u \log u
        \geq (u - 1) + \frac{1}{2} \left( \min_{v \in I} \ell''(v) \right) (u-1)^2
        = (u - 1) + \frac{(u-1)^2}{2 \max\{1, u\}} ,
    \]
    where $I = [\min\{1, u\}, \max\{1, u\}]$.
    Therefore
    \begin{align*}
        \kl{\mu_1}{\mu_2} = \int_{\real^d} (f \, \log f)(u)  \,  \mu_2(\d u)
        \geq \int_{\real^d} (f(u) - 1) \, \mu_2(\d u) + \frac{1}{2} \int_{\real^d} \frac{\lvert f(u) - 1 \rvert^2}{\max\{1, f(u)\}} \, \mu_2(\d u),
    \end{align*}
    Let $\theta(u) = \max\{1, f(u)\}$.
    The first term on the right-hand side is zero, and so we have that
    \begin{align*}
        d_g(\mu_1, \mu_2)^2
        &= \left( \int_{\real^d} g(u) \left\lvert  f(u) - 1 \right \rvert \, \mu_2(\d u) \right)^2
        \leq \int_{\real^d} \lvert g(u) \rvert^2 \theta(u)  \, \mu_2(\d u) \, \int_{\real^d} \frac{\left\lvert  f(u) - 1 \right \rvert^2}{\theta(u)} \, \mu_2(\d u) \\
        &\leq \int_{\real^d} \lvert g(u) \rvert^2 \bigl(f(u) + 1\bigr)  \, \mu_2(\d u) \, 2 \, \kl{\mu_1}{\mu_2}
        = 2 \Bigl( \mu_1\left[g^2\right] + \mu_2\left[g^2\right] \Bigr) \, \kl{\mu_1}{\mu_2},
    \end{align*}
    which concludes the proof.
\end{proof}

\begin{lemma}
    \label{lemma:elementary_gaussians}
    For all $d \in \nat^+$ and $\alpha > 0$,
    there is $C > 0$ such that
    for all $x_0, x_1, m_0, m_1 \in \real^d$
    and all symmetric positive definite matrices $S_0, S_1 \in \real^{d \times d}$ satisfying
    \[
        S_0 \succcurlyeq \frac{1}{\alpha} I_d, \qquad S_1 \succcurlyeq \frac{1}{\alpha} I_d,
    \]
    it holds that
    \begin{align}
        \notag
        (2\pi)^{\frac{d}{2}} \bigl\lvert g_0(x_0) - g_1(x_1) \bigr\rvert
        &\leq
        \frac{\alpha^{\frac{1+d}{2}}}{\sqrt{\e}} \lvert x_1 - x_0 - m_1 + m_0 \rvert + \frac{\alpha^{1 + \frac{d}{2}}}{\e} \lVert S_1 - S_0 \rVert \\
        \label{eq:auxiliary_gaussians_precise}
        &\qquad \qquad  + \alpha^d \left\lvert \sqrt{\det S_1} - \sqrt{\det S_0} \right\rvert.
    \end{align}
    Here $g_0$ and $g_1$ denote the densities of $\normal (m_0, S_0)$ and $\normal(m_1, S_1)$, respectively.
\end{lemma}
\begin{proof}
    For $s \in [0, 1]$,
    let
    \(
    x_s = (1-s) x_0 + s x_1,
    \)
    as well as
    \(
    m_s = (1-s) m_0 + s m_1
    \)
    and
    \(
    S_s = (1-s) S_0 + s S_1.
    \)
    Let us also introduce the non-normalized densities
    \[
        \widetilde g_0 = \exp \left( - \frac{1}{2} \bigl\lvert x - m_0 \bigr\rvert_{S_0}^2 \right), \qquad
        \widetilde g_1 = \exp \left( - \frac{1}{2} \bigl\lvert x - m_1 \bigr\rvert_{S_1}^2 \right),
    \]
    and let
    \(
    \lambda(s) \coloneq \exp \left( - \frac{1}{2} \bigl\lvert x_s - m_s \bigr\rvert_{S_s}^2 \right).
    \)
    It holds that
    \begin{align*}
        \widetilde g_1(x_1) - \widetilde g_0(x_0)
        = \int_{0}^{1} \frac{\d \lambda}{\d s}(s) \, \d s
        &= - \int_{0}^{1} \Bigl[  (x_1 - x_0 - m_1 + m_0)^\t S_s^{-1} (x_s - m_s) \\
        & \qquad \qquad \quad - \, \frac{1}{2} (x_s - m_s)^\t S_s^{-1}(S_1 - S_0) S_s^{-1} (x_s - m_s) \Bigr] \lambda(s) \, \d s.
    \end{align*}
    Therefore,
    using that $\max_{z \in \real} \left\lvert z \e^{-\frac{z^2}{2}} \right\rvert = \frac{1}{\sqrt{\e}}$ and  $\max_{z \in \real} \left\lvert z^2 \e^{-\frac{z^2}{2}} \right\rvert = \frac{2}{\e}$,
    we deduce that
    \begin{align*}
        \bigl\lvert \widetilde g_1(x_1) - \widetilde g_0(x_0) \bigr\rvert
        &\leq  \int_{0}^1 \Bigl[ \lvert x_1 - x_0 - m_1 + m_0 \rvert_{S_s}  \left\lvert x_s - m_s \right\rvert_{S_s}  \\
        &\qquad \qquad + \frac{1}{2} \left\lVert S_s^{-{\frac{1}{2}}}(S_1 - S_0)S_s^{-{\frac{1}{2}}} \right\rVert \left\lvert x_s - m_s \right\rvert_{S_s}^2 \Bigr] \lambda(s) \, \d s \\
        &\leq \sqrt{\frac{\alpha}{\e}} \lvert x_1 - x_0 - m_1 + m_0 \rvert
        + \frac{\alpha}{\e} \lVert S_1 - S_0 \rVert.
    \end{align*}
    Applying the triangle inequality
    \[
        \bigl\lvert g_0(x_0) - g_1(x_1) \bigr\rvert
        \leq \frac{\bigl\lvert \widetilde g_0(x_0) - \widetilde g_1(x_1) \bigr\rvert}{\sqrt{(2 \pi)^d \det S_0}}
        + \left\lvert \frac{1}{\sqrt{(2 \pi)^d \det S_0}} -  \frac{1}{\sqrt{(2 \pi)^d \det S_1}} \right\rvert  \bigl\lvert \widetilde g_1(x_1) \bigr\rvert,
    \]
    we obtain~\eqref{eq:auxiliary_gaussians_precise},
    which concludes the proof.
\end{proof}

\begin{lemma}
    \label{lemma:gaussians_relation_1_and_inf_norms}
    Denote by $g(\placeholder; \vect m, \vect S)$ the Lebesgue density of $\normal(\vect m, \mat S)$
    and by $\mathcal S^K_{\alpha}$ the set of symmetric~$K \times K$ matrices~$\mat M$ satisfying
    \begin{equation}
        \label{eq:bnd}
        \frac{1}{\alpha} \mat I_K \preccurlyeq \mat M \preccurlyeq \alpha \mat I_K.
    \end{equation}
    Then for all $\alpha \geq 1$, there exists $L_{\alpha} > 0$ such that
    for all parameters $(c_1, \vect m_1, \mat S_1) \in \real \times \real^K \times \mathcal S^K_{\alpha}$ and $(c_2, \vect m_2, \mat S_2) \in \real \times \real^K \times \mathcal S^K_{\alpha}$,
    \begin{equation}
        \label{eq:gaussian_auxiliary_lemma}
        \norm{\mathfrak h}_{\infty} \leq L_{\alpha} \norm{\mathfrak h}_{1},
        \qquad \mathfrak h(y) = c_1 g(y; \vect m_1, \mat S_1) - c_2 g(y; \vect m_2, \mat S_2).
    \end{equation}
\end{lemma}
\begin{remark}
    When $c_2=0$ then equation~\eqref{eq:gaussian_auxiliary_lemma} may be viewed as an inverse inequality.
    It then simply states that the $L^\infty$ norm of the density of a normal random variable is bounded from above by the~$L^1$ norm,
    uniformly for all densities from the set of Gaussians with a covariance matrix satisfying~\eqref{eq:bnd}.
\end{remark}

\begin{proof}
    For conciseness, let
    \(
    g_1(\placeholder) = g(\placeholder; \vect m_1, \mat S_1)
    \)
    and
    \(
    g_2(\placeholder) = g(\placeholder; \vect m_2, \mat S_2).
    \)
    \paragraph{Step 1. Simplification}
    It is sufficient to prove the statement for $c_1 = c_2 = 1$.
    Indeed, suppose that there is $\widetilde L_{\alpha} > 0$ such that
    \begin{equation}
        \label{eq:first_simple}
        \forall  (\vect m_1, \mat S_1) \in \real^K \times \mathcal S^K_{\alpha}, \qquad
        \forall  (\vect m_2, \mat S_2) \in \real^K \times \mathcal S^K_{\alpha}, \qquad
        \norm{g_1 - g_2}_{\infty} \leq \widetilde L_{\alpha} \norm{g_1 - g_2}_{1}.
    \end{equation}
    Then, by the triangle inequality,
    it holds that
    \begin{align*}
        \bigl\lVert c_1 g_1 - c_2 g_2 \bigr\rVert_{\infty}
        &\leq
        \lvert c_1 \rvert \bigl\lVert  g_1 - g_2 \bigr\rVert_{\infty} +
        \lvert c_1 - c_2 \rvert \bigl\lVert  g_2 \bigr\rVert_{\infty} \\
        &\leq \lvert c_1 \rvert \widetilde L_{\alpha} \bigl\lVert  g_1 - g_2 \bigr\rVert_{1}
        + \lvert c_1 - c_2 \rvert \bigl\lVert  g_2 \bigr\rVert_{\infty} \\
        &\leq  \widetilde L_{\alpha} \bigl\lVert  c_1 g_1 - c_2 g_2 \bigr\rVert_{1}
        + \lvert c_2 - c_1 \rvert \bigl\lVert  g_2 \bigr\rVert_{1}
        + \lvert c_2 - c_1 \rvert \bigl\lVert  g_2 \bigr\rVert_{\infty}.
    \end{align*}
    By Jensen's inequality,
    it holds that
    \[
        \lvert c_1 - c_2 \rvert = \left\lvert \int_{\real^d} c_1 g_1(x) - c_2 g_2(x) \, \d x \right\rvert
        \leq \bigl\lVert  c_1 g_1 - c_2 g_2 \bigr\rVert_{1},
    \]
    and so we deduce that
    \[
        \bigl\lVert c_1 g_1 - c_2 g_2 \bigr\rVert_{\infty}
        \leq  \bigl\lVert  c_1 g_1 - c_2 g_2 \bigr\rVert_{1} \left( \widetilde L_{\alpha} + 1 + \sqrt{\frac{\alpha^K}{(2\pi)^K}} \right).
    \]
    Furthermore, since both sides of the inequality~\eqref{eq:first_simple} are invariant under translation,
    it is sufficient to consider the case where $\vect m_2 = \vect 0$,
    which we do from now on.
    Finally, note that
    \begin{align*}
        g_1(y) - g_2(y)
        = \frac{1}{\sqrt{\det \mat S_2}} \left( g\left(\sqrt{\mat S_2^{-1}}y; \sqrt{\mat S_2^{-1}} \vect m_1, \sqrt{\mat S_2^{-1}}\mat S_1\sqrt{\mat S_2^{-1}}\right) - g\left(\sqrt{\mat S_2^{-1}}y; \vect 0, \mat I_K\right) \right),
    \end{align*}
    and so we can also assume without loss of generality that $\mat S_2 = \mat I_K$.
    Indeed, assume that the inequality~\eqref{eq:first_simple} is satisfied in this particular case with a constant $\widehat L_{\alpha}$.
    Since
    \[
        \frac{1}{\alpha^2} \mat I_K
        \preccurlyeq \frac{1}{\alpha \norm{\mat S_2}}\mat I_K
        \preccurlyeq \sqrt{\mat S_2^{-1}}\mat S_1\sqrt{\mat S_2^{-1}}
        \preccurlyeq \alpha \norm{\mat S_2^{-1}} \mat I_K
        \preccurlyeq \alpha^2 \mat I_K,
    \]
    we deduce that if $(\mat S_1, \mat S_2) \in \mathcal S^{K}_{\alpha} \times S^{K}_{\alpha}$ for some $\alpha > 0$,
    then
    \[
        \left(\sqrt{\mat S_2^{-1}}\mat S_1\sqrt{\mat S_2^{-1}}, \mat I_K\right) \in \mathcal S^K_{\alpha^2} \times \mathcal S^K_{\alpha^2}.
    \]
    Therefore, using the change of variable $y \mapsto \sqrt{\mat S_2^{-1}}y$
    together with~\eqref{eq:first_simple} in the particular case~$\mat S_2 = \mat I_K$,
    we obtain that
    \begin{align*}
        &\norm{g(\placeholder; \vect m_1, \mat S_1) - g(\placeholder; \vect 0, \mat S_2)}_{\infty} \\
        &\qquad = \frac{1}{\sqrt{\det \mat S_2}} \left\lVert  g\left(\placeholder; \sqrt{\mat S_2^{-1}} \vect m_1, \sqrt{\mat S_2^{-1}}\mat S_1\sqrt{\mat S_2^{-1}}\right) - g\left(\placeholder; \vect 0, \mat I_K\right) \right\rVert_{\infty} \\
        &\qquad \leq  \frac{\widehat L_{\alpha^2} }{\sqrt{\det \mat S_2}} \norm*{g\left(\placeholder; \sqrt{\mat S_2^{-1}} \vect m_1, \sqrt{S_2^{-1}}S_1\sqrt{S_2^{-1}}\right) -  g\left(\placeholder; \vect 0, \mat I_K\right)}_{1} \\
        &\qquad = \frac{\widehat L_{\alpha^2}}{\sqrt{\det \mat S_2}} \norm*{g(\placeholder; \vect m_1, \mat S_1) - g(\placeholder; \vect 0, \mat S_2)}_{1}
        \leq \alpha^{\frac{K}{2}}\widehat L_{\alpha^2} \norm*{g(\placeholder; \vect m_1, \mat S_1) - g(\placeholder; \vect 0, \mat S_2)}_{1},
    \end{align*}
    and so the bound~\eqref{eq:first_simple} is valid in general with $\widetilde L_{\alpha} = \alpha^{\frac{K}{2}} \widehat L_{\alpha^2}$.

    \paragraph{Step 2. Proof of the simplified statement}
    It remains to show that there exists for all $\alpha > 0$ a constant~$\widehat L_{\alpha}$ such that
    the following inequality holds for all $(\vect m, \mat S) \in \real^K \times \mathcal S^{K}_{\alpha}$:
    \begin{equation}
        \label{eq:simplified_case}
        \norm{\mathfrak h}_{\infty} \leq \widehat L_{\alpha} \norm{\mathfrak h}_{1},
        \qquad \mathfrak h(y) = g(y; \vect m, \mat S) - g(y; \vect 0, \mat I_K).
    \end{equation}
    To this end, fix $\alpha > 0$, fix $\varepsilon \in (0, 1)$
    and assume first that $\norm{\mat S - \mat I_K} \geq \varepsilon$.
    By the lower bound on the total variation distance between Gaussians in~\cite[Proposition 2.2]{2018arXiv181008693D},
    we have that
    \begin{align*}
        \frac{1}{2}\norm{\mathfrak h}_1
        &\geq 1 - \frac{\det \mat S^{1/4} \det \mat I_K^{1/4}}{\det \left(\frac{\mat S + \mat I_K}{2} \right)^{1/2}} \exp \left( - \frac{1}{8} \vect m^\t \left( \frac{\mat S + \mat I_K}{2} \right)^{-1} \vect m \right).
    \end{align*}
    In particular, it holds that
    \begin{align}
        \label{eq:case_1_1}
        \frac{1}{2}\norm{\mathfrak h}_1
        &\geq 1 - \frac{\det \mat S^{1/4} \det \mat I_K^{1/4}}{\det \left(\frac{\mat S + \mat I_K}{2} \right)^{1/2}}
        = 1 - \prod_{i=1}^K \sqrt{\frac{\sqrt{\lambda_{i}}}{\frac{\lambda_{i} + 1}{2}}}
        \geq 1 - \sqrt{\frac{2\sqrt{1 + \varepsilon}}{2 + \varepsilon}},
    \end{align}
    where $(\lambda_{i})_{i=1}^{K}$ are the eigenvalues of $\mat S$.
    In the last inequality, we used that,
    since all the terms in the product are bounded from above by 1 by the arithmetic mean-geometric mean inequality,
    and since at least one eigenvalue is not in the interval $(1 - \varepsilon, 1 + \varepsilon)$
    given that $\norm{\mat S - \mat I_K} \geq \varepsilon$,
    it holds that
    \[
        \prod_{i=1}^K \sqrt{\frac{\sqrt{\lambda_{i}}}{\frac{\lambda_{i} + 1}{2}}}
        \leq
        \max_{|\lambda - 1| \geq \varepsilon}  \sqrt{\frac{\sqrt{\lambda}}{\frac{\lambda + 1}{2}}}
        = \sqrt{\frac{2\sqrt{1 + \varepsilon}}{2 + \varepsilon}}.
    \]
    On the other hand,
    since $S \succcurlyeq \frac{1}{\alpha} I_K$,
    it holds that $\det S \geq \frac{1}{\alpha^K}$,
    and so
    \begin{equation}
        \label{eq:case_1_2}
        \norm{\mathfrak h}_{\infty} \leq 2 \left( \frac{\alpha}{2\pi} \right)^{\frac{K}{2}}.
    \end{equation}
    Combining~\eqref{eq:case_1_1} and~\eqref{eq:case_1_2} gives that
    \[
        \norm{\mathfrak h}_1
        \geq C_1 \norm{\mathfrak h}_{\infty}, \qquad
        C_1 :=
        \left( 1 - \sqrt{\frac{2\sqrt{1 + \varepsilon}}{2 + \varepsilon}} \right)
        \left(\frac{2\pi}{\alpha}\right)^{\frac{K}{2}}.
    \]

    Consider now the case where $\norm{\mat S - \mat I_K} \leq \varepsilon$.
    Since $S \mapsto \sqrt{\det S}$ is Lipschitz continuous
    over the set of symmetric positive definite matrices $S$ such that $\norm{\mat S - \mat I_K} \leq \varepsilon$,
    with a Lipschitz constant we denote by~$c_{\varepsilon}$,
    it holds by~\cref{lemma:elementary_gaussians} that
    \begin{align*}
        \norm{\mathfrak h}_{\infty}
            &\leq
            (2\pi)^{- \frac{K}{2}}
            \left(
                \frac{\alpha^{\frac{1+K}{2}}}{\sqrt{\e}} \lvert m \rvert + \frac{\alpha^{1 + \frac{K}{2}}}{\e} \lVert \mat S - \mat I_K \rVert
                + \alpha^K \left\lvert \sqrt{\det S} - \sqrt{\det \mat I_K} \right\rvert
            \right) \\
            &\leq
            (2\pi)^{- \frac{K}{2}}
            \left(
                \frac{\alpha^{\frac{1+K}{2}}}{\sqrt{\e}} \lvert m \rvert
            + \left( \frac{\alpha^{1 + \frac{K}{2}}}{\e} + c_{\varepsilon} \alpha^K \right) \lVert \mat S - \mat I_K \rVert \right).
    \end{align*}
    In view of~\eqref{eq:case_1_2},
    this implies that there exists a constant $C$ depending only on~$(\varepsilon, \alpha, K)$ such that
    \begin{equation}
        \label{eq:linf_bound}
        \norm{\mathfrak h}_{\infty}
        \leq C \Bigl( \min \bigl\{ \lvert m \rvert, 1 \bigr\}  + \lVert \mat S - \mat I_K \rVert \Bigr).
    \end{equation}
    On the other hand,
    since the characteristic function of $\mathcal N(m, \mat S)$
    is given by $u \mapsto \e^{{\rm i} m^\t u - \frac{1}{2} u^\t \mat S u}$,
    it holds by definition of the characteristic function that
    \[
        \forall u \in \real^K, \qquad
        \e^{{\rm i} m^\t u - \frac{1}{2} u^\t \mat S u} - \e^{-\frac{\lvert u \rvert^2}{2}}
        = \int_{\real^K} \e^{{\rm i} u^\t x } \mathfrak h(x) \, \d x.
    \]
    Therefore, it holds that
    \begin{align*}
        \lVert \mathfrak h \rVert_1
            &\geq
            \sup_{|u| \leq 1} \left\lvert  \e^{{\rm i} m^\t u - \frac{1}{2} u^\t \mat S u} - \e^{-\frac{\lvert u \rvert^2}{2}} \right \rvert
            = \sup_{|u| \leq 1} \left\lvert  \e^{- \frac{1}{2} u^\t \mat S u} - \e^{- {\rm i} m^\t u -\frac{\lvert u \rvert^2}{2}} \right \rvert.
    \end{align*}
    It is clear from elementary geometry in the complex plane that
    \[
        \forall u \in \real^K, \qquad
        \left\lvert  \e^{- \frac{1}{2} u^\t \mat S u} - \e^{- {\rm i} m^\t u -\frac{\lvert u \rvert^2}{2}} \right \rvert
        \geq
        \max
        \left\{
            \left\lvert  \sin(m^\t u) \right\rvert \e^{-\frac{\lvert u \rvert^2}{2}} ,
            \left\lvert  \e^{- \frac{1}{2} u^\t \mat S u} - \e^{-\frac{\lvert u \rvert^2}{2}} \right \rvert
        \right\},
    \]
    and so we have
    \begin{align*}
        \lVert \mathfrak h \rVert_1
            &\geq
            \max
            \left\{
                \sup_{|u| \leq 1}
                \left\lvert  \sin(m^\t u) \right\rvert \e^{-\frac{\lvert u \rvert^2}{2}} ,
                \sup_{|u| \leq 1}
                \left\lvert  \e^{- \frac{1}{2} u^\t \mat S u} - \e^{-\frac{\lvert u \rvert^2}{2}} \right \rvert
            \right\} \\
            &\geq
            \max
            \left\{
                \e^{-\frac{1}{2}} \sup_{|u| \leq 1}
                \left\lvert  \sin(m^\t u) \right\rvert ,
                \frac{\e^{-1}}{2}
                \sup_{|u| \leq 1}
                \left\lvert  u^\t (\mat S - \mat I_K) u  \right \rvert
            \right\} \\
            &=
            \max
            \left\{
                \e^{-\frac{1}{2}} \sup_{|u| \leq 1}
                \left\lvert  \sin(m^\t u) \right\rvert ,
                \frac{\e^{-1}}{2}
                \lVert \mat S - \mat I_K \rVert
            \right\}.
    \end{align*}
    In the second inequality, we used that $u^\t S u \leq 1 + \varepsilon \leq 2$ for all $|u| \leq 1$,
    together with the elementary inequality $\left\lvert \e^a - \e^b \right\rvert \geq \e^{\min\{a, b\}} |a - b|$.
    To conclude,
    considering the particular value
    \[
        u = \frac{m}{|m| \max\left\{ 1, \frac{4|m|}{\pi} \right\}},
    \]
    we obtain that
    \[
        \sup_{|u| \leq 1} \left\lvert  \sin(m^\t u) \right\rvert
        \geq \sin\left(\min \left\{|m|, \frac{\pi}{4} \right\}\right)
        = \int_{0}^{\min \left\{|m|, \frac{\pi}{4} \right\}}
        \cos (t) \, \d t
        \geq \cos \left(\frac{\pi}{4}\right) \min \left\{|m|, \frac{\pi}{4} \right\}.
    \]
    Thus, using that $\max\{A, B\} \geq \frac{A}{2} + \frac{B}{2}$,
    we obtain
    \begin{align}
        \notag
        \lVert \mathfrak h \rVert_1
            &\geq
            \max
            \left\{
                \e^{-\frac{1}{2}} \cos\left(\frac{\pi}{4}\right)
                \min \left\{ |m|, \frac{\pi}{4} \right\},
                \frac{\e^{-1}}{2}
                \lVert \mat S - \mat I_K \rVert
            \right\} \\
            \label{eq:l1_bound}
            &\geq \frac{1}{2} \min \Bigl\{\e^{-\frac{1}{2}} \cos\left(\frac{\pi}{4}\right) \frac{\pi}{4}, \frac{\e^{-1}}{2} \Bigr\}
            \Bigl( \min \left\{ \lvert m \rvert, 1 \right\}  + \lVert \mat S - \mat I_K \rVert \Bigr)
    \end{align}
    Combining~\eqref{eq:l1_bound} with~\eqref{eq:linf_bound} leads to
    $\lVert \mathfrak h \rVert_{1} \geq C_2 \lVert \mathfrak h \rVert_{\infty}$,
    which concludes the proof of the case $\norm{\mat S - \mat I_K} \leq \varepsilon$.
    Consequently, the statement~\eqref{eq:simplified_case} holds in general with constant $\widehat L_{\alpha} = \min\{C_1, C_2\}^{-1}$.
\end{proof}

\begin{lemma}
    \label{lemma:lipschitz_density}
    Let~$\op P$ and $\op Q$ denote the operators on probability measures given respectively in~\eqref{eq:Markov_kernel} and~\eqref{eq:definition_Q}.
    Suppose that~\cref{assumption:ensemble_kalman} is satisfied and that $|h|_{C^{0,1}} \leq \ell_{\vect h} < \infty$.
    Then there is~$L = L(\kappa_{\vect \Psi}, \kappa_{\vect h}, \ell_{\vect h}, \mat \Sigma, \mat \Gamma)$ such that
    for all $(u_1, u_2, y) \in \real^d \times \real^d \times \real^K$ and all $\mu \in \mathcal P(\real^d)$
    the density of $p = \op Q \op P \mu$ satisfies
    \begin{equation}
        \label{eq:lipschitz_density}
        \lvert p(u_1, y) - p(u_2, y) \rvert \leq L \lvert u_1 - u_2 \rvert
        \exp \left(- \frac{1}{4} \left( \min\Bigl\{\vecnorm{u_1}_{\mat \Sigma}^2, \vecnorm{u_2}_{\mat \Sigma}^2\Bigr\} + \vecnorm{y}_{\mat \Gamma}^2\right) \right).
    \end{equation}
\end{lemma}
\begin{proof}
    Throughout this proof, $C$ denotes a constant whose value is irrelevant in the context,
    depends only on $\kappa_{\vect \Psi}, \kappa_{\vect h}, \ell_{\vect h}, \mat \Sigma, \mat \Gamma$,
    and may change from line to line.
    Since the function $g(x):= \e^{-x^2}$ has derivative~$-2 x \e^{-x^2}$
    and since $\lvert x \e^{-x^2} \rvert \leq \e^{-\frac{2x^2}{3}}$ for all $x \in \real$,
    it holds for all $(a, b) \in \real^2$ that there is $\xi$ between~$|a|$ and $|b|$ such that
    \begin{align}
        \label{eq:elementary_inequality_gaussian_density}
        \left\lvert \e^{-a^2} - \e^{-b^2} \right\rvert
        =  \abs{b - a} \, \lvert g'(\xi) \rvert
        \leq 2 \abs{b - a} \left( \e^{-\frac{2a^2}{3}} + \e^{-\frac{2b^2}{3}} \right).
    \end{align}
    Using this inequality with $a^2 = \frac{1}{2} \vecnorm*{u_1 - \vect \Psi(v)}_{\mat \Sigma}^2$ and $b^2 = \frac{1}{2} \vecnorm*{u_2 - \vect \Psi(v)}_{\mat \Sigma}^2$,
    and then using the triangle inequality,
    we deduce that, for all $(u_1, u_2, v) \in \real^{d} \times \real^{d} \times \real^{d}$,
    \begin{align*}
        &\left\lvert \exp \left( - \frac{1}{2} \vecnorm*{u_1 - \vect \Psi(v)}_{\mat \Sigma}^2 \right) -\exp \left( - \frac{1}{2} \vecnorm*{u_2 - \vect \Psi(v)}_{\mat \Sigma}^2 \right) \right\rvert \\
        \label{eq:initial_inequality}
        & \hspace{2cm} \leq C \vecnorm{u_2 - u_1}_{\mat \Sigma} \left( \exp \left( - \frac{1}{3} \vecnorm*{u_1 - \vect \Psi(v)}_{\mat \Sigma}^2 \right) + \exp \left( - \frac{1}{3} \vecnorm*{u_2 - \vect \Psi(v)}_{\mat \Sigma}^2 \right) \right).
    \end{align*}
    Integrating out the $v$ variable with respect to $\mu$ and using the equivalence of norms,
    we obtain that
    \begin{align*}
        \lvert \op P \mu(u_1) - \op P \mu(u_2) \rvert
        &\leq
        C \abs{u_1 - u_2} \int_{\real^{d}} \exp \left( - \frac{1}{3} \vecnorm*{u_1 - \vect \Psi(v)}_{\mat \Sigma}^2 \right) \, \mu(\d v) \\
        & \quad + C \abs{u_1 - u_2} \int_{\real^{d}} \exp \left( - \frac{1}{3} \vecnorm*{u_2 - \vect \Psi(v)}_{\mat \Sigma}^2 \right) \, \mu(\d v).
    \end{align*}
    By Young's inequality, it holds for all $\delta > 0$ that
    \begin{equation}
        \label{eq:young}
        \forall (a, b) \in \real^d \times \real^d, \qquad
        \vecnorm*{a - b}_{\mat \Sigma}^2
        \geq \frac{1}{1 + \delta} \vecnorm*{a}_{\mat \Sigma}^2 -
        \frac{1}{\delta} \vecnorm*{b}_{\mat \Sigma}^2.
    \end{equation}
    Using this inequality with $\delta = \frac{1}{3}$ together with the assumption that $\vect \Psi$ is bounded,
    we deduce that
    \begin{align}
        \notag
        \lvert \op P \mu(u_1) - \op P \mu(u_2) \rvert
        &\leq
        C \abs{u_1 - u_2} \int_{\real^{d}} \exp \left( - \frac{1}{4} \vecnorm*{u_1}_{\mat \Sigma}^2 \right) \, \mu(\d v)
        + C \abs{u_1 - u_2} \int_{\real^{d}} \exp \left( - \frac{1}{4} \vecnorm*{u_2}_{\mat \Sigma}^2 \right) \, \mu(\d v) \\
        \label{eq:techincal_lemma_bound1}
        &\leq 2 C \abs{u_1 - u_2} \exp \left( - \frac{1}{4}  \min\Bigl\{\vecnorm{u_1}_{\mat \Sigma}^2, \vecnorm{u_2}_{\mat \Sigma}^2\Bigr\} \right).
    \end{align}

    Next, letting $h_s = (1-s) h(u_1) + s h(u_2)$
    and using a reasoning similar to that in the proof of~\cref{lemma:elementary_gaussians},
    we obtain the following inequalities,
    which hold for all $(u_1, u_2, y) \in \real^d \times \real^d \times \real^K$:
    \begin{align}
        \notag
        \Bigl\lvert \normal\bigl(\vect h(u_1), \mat \Gamma\bigr)(y) - \normal\bigl(\vect h(u_2), \mat \Gamma\bigr)(y) \Bigr\rvert
        &\leq \vecnorm{h(u_2) - h(u_1)}_{\mat \Gamma}  \int_{0}^{1} \vecnorm{h_s - y}_{\mat \Gamma} \exp \left( - \frac{1}{2} \vecnorm*{h_s - y}_{\mat \Gamma}^2 \right) \, \d s \\
        \label{eq:techincal_lemma_bound2}
        &\leq C \vecnorm{u_2 - u_1}  \exp \left( - \frac{1}{4}  \vecnorm{y}_{\mat \Gamma}^2 \right),
    \end{align}
    where we  used the Lipschitz continuity of $h$,
    together with~\eqref{eq:young} and the boundedness of~$\vect h$, in the last inequality.
    In order to conclude the proof,
    using the definition of~$\op P$,
    we calculate that
    \begin{align*}
        p(u_1, y) - p(u_2, y)
        &= \op P \mu(u_1) \, \normal\bigl(\vect h(u_1), \mat \Gamma\bigr) (y) - \op P \mu(u_2) \, \normal\bigl(\vect h(u_2), \mat \Gamma\bigr) (y) \\
        &= \bigl(\op P \mu(u_1) - \op P \mu(u_2)\bigr) \, \normal\bigl(\vect h(u_1), \mat \Gamma\bigr) (y) \\
        &\quad + \op P \mu(u_2) \Bigl( \normal\bigl(\vect h(u_1), \mat \Gamma\bigr) - \normal\bigl(\vect h(u_2), \mat \Gamma\bigr) (y) \Bigr).
    \end{align*}
    The first and second terms on the right hand side can be bounded
    by using~\eqref{eq:techincal_lemma_bound1} and~\eqref{eq:techincal_lemma_bound2}, respectively,
    leading to~\eqref{eq:lipschitz_density}.
\end{proof}

\section{Technical Results for Theorems \ref{theorem:main_theorem} and \ref{theorem:main_theorem_gaussian_projection}}
\label{appendix:A}

We show moment bounds in~\cref{sub:moment_bounds},
and we prove that $\op T_j\mu = \op B_j \mu$ for any Gaussian probability measure~$\mu$ in~\cref{sub:mean_field_is_conditioning_for_gaussians}.
Finally, we prove the stability results used in the proofs of~\cref{theorem:main_theorem,theorem:main_theorem_gaussian_projection} in~\cref{sub:stability_results,sub:additional_stability_results},
respectively.
\subsection{Moment Bounds}
\label{sub:moment_bounds}
\begin{lemma}
    [Moment bounds]
    \label{lemma:bound_on_Pmu}
    Let~$\mu$ denote a probability measure on~$\real^d$.
    Under~\cref{assumption:ensemble_kalman},
    it holds that
    \begin{equation}
        \label{eq:bound_on_Pmu}
        \bigl\lvert \mathcal M(\op P \mu) \bigr\rvert \leq \kappa_{\vect \Psi}, \qquad
        \mat \Sigma \preccurlyeq \mathcal C(\op P \mu) \preccurlyeq \kappa_{\vect \Psi}^2 \mat I_d + \mat \Sigma.
    \end{equation}
\end{lemma}
\begin{proof}
From the definition of~$\op P$ in~\eqref{eq:Markov_kernel},
we have that
\begin{align*}
    \mathcal M(\op P \mu)
    = \int_{\real^d} u \, \op P \mu(u) \, \d u
    &= \frac{1}{\sqrt{(2\pi)^d \det \mat \Sigma}} \int_{\real^d} \int_{\real^d} u \exp \left( - \frac{1}{2} \vecnorm*{u - \vect \Psi(v)}_{\mat \Sigma}^2 \right) \, \mu(\d v) \, \d u \\
    &= \int_{\real^d} \vect \Psi(v) \, \mu(\d v),
\end{align*}
where the last equality is obtained by changing the order of integration using Fubini's theorem.
Using the first item in~\cref{assumption:ensemble_kalman},
we then deduce the first inequality in~\eqref{eq:bound_on_Pmu}.
For the second inequality in~\eqref{eq:bound_on_Pmu}, we first note the following inequality
which holds for any $m, v \in \real^d$ by~\cref{lemma:auxiliary_ineq_second_moment}:
\begin{align}
    \label{eq:ubnd}
    &\int_{\real^d} (u - \vect m) \otimes (u - \vect m) \exp \Bigl( - \frac{1}{2} \vecnorm*{u - \vect \Psi(v)}_{\mat \Sigma}^2 \Bigr) \, \d u \nonumber\\
    &\qquad \succcurlyeq
    \int_{\real^d} \bigl(u - \vect \Psi(v)\bigr) \otimes \bigl(u - \vect \Psi(v)\bigr) \exp \Bigl( - \frac{1}{2} \vecnorm*{u - \vect \Psi(v)}_{\mat \Sigma}^2 \Bigr) \, \d u;
\end{align}
The result follows by using the fact that $\Psi(v)$ is the mean under Gaussian $\mathcal{N}\bigl(\Psi(v), \mat \Sigma\bigr).$
Now choose $m$ to be the mean under measure $\op P\mu$ and note that, by conditioning on $v$
and using~\eqref{eq:ubnd},
\begin{align*}
    \mathcal C(\op P \mu)&=\int_{\real^d} (u-\vect m) \otimes (u - \vect m) \,  \op P \mu(u) \, \d u  \\
                         & = \frac{1}{\sqrt{(2\pi)^d \det \mat \Sigma}} \int_{\real^d} \left(\int_{\real^d} (u - \vect m) \otimes (u - \vect m) \exp \Bigl( - \frac{1}{2} \vecnorm*{u - \vect \Psi(v)}_{\mat \Sigma}^2 \Bigr) \, \d u \right) \, \mu(\d v) \\
                         & \succcurlyeq \frac{1}{\sqrt{(2\pi)^d \det \mat \Sigma}} \int_{\real^d} \left(\int_{\real^d} \bigl(u - \vect \Psi(v)\bigr) \otimes \bigl(u - \vect \Psi(v)\bigr) \exp \Bigl( - \frac{1}{2} \vecnorm*{u - \vect \Psi(v)}_{\mat \Sigma}^2 \Bigr) \, \d u\right)  \, \mu(\d v) \\
                         & = \int_{\real^d} \mat \Sigma \, \mu(\d v) = \mat \Sigma,
\end{align*}
and so~$\mathcal C(\op P \mu) \succcurlyeq \mat \Sigma$.
On the other hand,
using~\cref{lemma:auxiliary_ineq_second_moment} again together with the fact that,
by the Cauchy-Schwarz inequality, $\vect a \vect a^\t \preccurlyeq (\vect a^\t \vect a) \mat I_{d}$ for any vector $\vect a \in \real^d$,
we have
\begin{align}
    \nonumber
    \mathcal C(\op P \mu)
    &\preccurlyeq \int_{\real^d} u \otimes u \,  \op P \mu(u) \, \d u  \\
    \nonumber
    &= \frac{1}{\sqrt{(2\pi)^d \det \mat \Sigma}} \int_{\real^d} \int_{\real^d} u \otimes u \exp \left( - \frac{1}{2} \vecnorm*{u - \vect \Psi(v)}_{\mat \Sigma}^2 \right) \mu(\d v) \, \d u \\
    \label{eq:covariance_Pmu}
    &= \int_{\real^d} \bigl( \vect \Psi(v) \otimes \vect \Psi(v) + \mat \Sigma \bigr) \, \mu(\d v)
    \preccurlyeq \kappa_{\vect \Psi}^2 \mat I_d + \mat \Sigma,
\end{align}
which concludes the proof.
\end{proof}

It is possible, by using a similar reasoning,
to obtain bounds on the moments of~$\op Q \op P \mu$.

\begin{lemma}
    \label{lemma:bound_on_QPmu}
    Let~$\mu$ denote a probability measure on~$\real^d$.
    Under~\cref{assumption:ensemble_kalman},
    it holds that
    \begin{equation}
        \label{eq:bound_on_QPmu_mean}
        \bigl\lvert \mathcal M(\op Q \op P \mu) \bigr\rvert
        \leq \sqrt{\kappa_{\vect \Psi}^2 + \kappa_{\vect h}^2}, \qquad
    \end{equation}
    and
    \begin{equation}
        \label{eq:bound_on_QPmu_covariance}
        \min \left\{ \frac{\gamma\sigma}{2 \kappa_{\vect h}^2+\gamma} , \frac{\gamma}{2} \right\}
        \mat I_{d + K}
        \preccurlyeq \mathcal C(\op Q \op P \mu) \preccurlyeq
        \begin{pmatrix}
            2 \kappa_{\vect \Psi}^2 \mat I_d  + 2 \mat \Sigma & \mat 0_{d \times K} \\
            \mat 0_{K \times d} & 2 \kappa_{\vect h}^2 \mat I_K + \mat \Gamma
        \end{pmatrix}.
    \end{equation}
\end{lemma}
\begin{proof}
The inequality~\eqref{eq:bound_on_QPmu_mean} follows immediately from~\cref{assumption:ensemble_kalman} and the fact that
\begin{align}
    \mathcal M(\op Q \op P \mu) =
    \begin{pmatrix}
        \mathcal M(\op P \mu) \\
        \op P \mu [\vect h]
    \end{pmatrix}.
\end{align}
For inequality~\eqref{eq:bound_on_QPmu_covariance},
let $\phi\colon \real^d \to \real^{d + K}$ denote the map $\phi(u) = \bigl(u, \vect h(u)\bigr)$,
and let $\phi_{\sharp}$ denote the associated pushforward map on measures.
A calculation gives
\begin{equation}
    \label{eq:decomposition_covariance}
    \mathcal C(\op Q \op P \mu) =
    \mathcal C(\phi_{\sharp} \op P \mu) +
    \begin{pmatrix}
        \mat 0_{d\times d} & \mat 0_{d \times K} \\
        \mat 0_{K \times d} & \mat \Gamma
    \end{pmatrix}
    =
    \begin{pmatrix}
        \mathcal C^{uu}(\phi_{\sharp} \op P \mu) & \mathcal C^{uy}(\phi_{\sharp} \op P \mu) \\
        \mathcal C^{yu}(\phi_{\sharp} \op P \mu) & \mathcal C^{yy}(\phi_{\sharp} \op P \mu) + \mat \Gamma
    \end{pmatrix}.
\end{equation}
For any $(\vect a, \vect b) \in \real^{d} \times \real^K$,
it holds that
\[
    2
    \begin{pmatrix}
        \vect a \vect a^\t & \mat 0_{d \times K} \\
        \mat 0_{K \times d} & \vect b \vect b^\t
    \end{pmatrix}
    -
    \begin{pmatrix}
        \vect a \\
        \vect b
    \end{pmatrix}
    \otimes
    \begin{pmatrix}
        \vect a \\
        \vect b
    \end{pmatrix}
    =
    \begin{pmatrix}
        \vect a \\
        -\vect b
    \end{pmatrix}
    \otimes
    \begin{pmatrix}
        \vect a \\
        -\vect b
    \end{pmatrix}
    \succcurlyeq \mat 0_{(d+K) \times (d+K)}.
\]
Therefore, we obtain
\[
    \begin{pmatrix}
        \vect a \\
        \vect b
    \end{pmatrix}
    \otimes
    \begin{pmatrix}
        \vect a \\
        \vect b
    \end{pmatrix}
    \preccurlyeq
    2
    \begin{pmatrix}
        \vect a \vect a^\t & \mat 0_{d \times K} \\
        \mat 0_{K \times d} & \vect b \vect b^\t
    \end{pmatrix}
\]
which enables to deduce, using~\cref{lemma:bound_on_Pmu} and~\cref{assumption:ensemble_kalman},
that
\begin{align}
    \mathcal C(\phi_{\sharp} \op P \mu)
    \nonumber
    &\preccurlyeq \int_{\real^d} \begin{pmatrix} u \\ \vect h(u) \end{pmatrix} \otimes \begin{pmatrix} u \\ \vect h(u) \end{pmatrix} \, \op P \mu(u) \, \d u  \\
    \label{eq:bound_covariance_pushforward}
    &\preccurlyeq  2\int_{\real^d}
    \begin{pmatrix}
        u u^\t & \mat 0_{d \times K} \\
        \mat 0_{K \times d} & \vect h(u) \vect h(u)^\t
    \end{pmatrix} \, \op P \mu(u) \, \d u
    \preccurlyeq
    2
    \begin{pmatrix}
        \mathcal \kappa_{\vect \Psi}^2 \mat I_d + \mat \Sigma & \mat 0_{d \times K} \\
        \mat 0_{K \times d} & \kappa_{\vect h}^2 \mat I_K
    \end{pmatrix}.
\end{align}
where we used~\eqref{eq:covariance_Pmu} in the last inequality.
Combined with~\eqref{eq:decomposition_covariance},
this inequality leads to the upper bound~\eqref{eq:bound_on_QPmu_covariance}.
For the lower bound,
note that by the Cauchy--Schwarz inequality,
it holds for any probability measure~$\pi \in \mathcal P(\real^d \times \real^K)$
and all $(\vect a, \vect b) \in \real^d \times \real^K$
\begin{align*}
    \vecnorm*{\vect a^\t \mathcal C^{uy}(\pi) \vect b}
    &= \int_{\real^d \times \real^K} \left(\vect a^\t \bigl(u - \mathcal M^u(\pi)\bigr)\right) \left(\vect b^\t \bigl(y - \mathcal M^y(\pi)\bigr)\right) \pi(\d u \d y) \\
    &\leq \sqrt{\vect a^\t \mathcal C^{uu}(\pi) \vect a} \, \sqrt{\vect b^\t \mathcal C^{yy}(\pi) \vect b}.
\end{align*}
Therefore, by Young's inequality, it holds for all~$\varepsilon \in (0,1)$ and for all $(\vect a, \vect b) \in \real^d \times \real^K$, that
\begin{align*}
    \begin{pmatrix}
        \vect a \\ \vect b
    \end{pmatrix}^\t
    \mathcal C(\phi_{\sharp} \op P \mu)
    \begin{pmatrix}
        \vect a \\ \vect b
    \end{pmatrix}
    &\geq (1 - \varepsilon) \vect a^\t \mathcal C^{uu}(\phi_{\sharp} \op P \mu) \vect a - \left(\frac{1}{\varepsilon} - 1\right) \vect b^\t \mathcal C^{yy}(\phi_{\sharp} \op P \mu) \vect b \\
    &\geq (1 - \varepsilon) \vect a^\t \mat \Sigma \vect a - \left(\frac{1}{\varepsilon} - 1\right) \kappa_{\vect h}^2 \vecnorm{\vect b}^2,
\end{align*}
where we employed~\eqref{eq:bound_on_Pmu} and the bound~$\mathcal C^{yy}(\phi_{\sharp} \op P \mu) \preccurlyeq \kappa_{H}^2 \mat I_K$ in the last inequality.
Using~\eqref{eq:decomposition_covariance},
we deduce that
\begin{align}
    \nonumber
    \begin{pmatrix}
        \vect a \\ \vect b
    \end{pmatrix}^\t
    \mathcal C(\op Q \op P \mu)
    \begin{pmatrix}
        \vect a \\ \vect b
    \end{pmatrix}
    &\geq (1 - \varepsilon) \vect a^\t \mat \Sigma \vect a - \left(\frac{1}{\varepsilon} - 1\right) \kappa_{\vect h}^2 \vecnorm*{\vect b}^2 +  \vect b^\t \mat \Gamma \vect b \\
    \label{eq:intermediate_moment}
    &\geq (1 - \varepsilon) \sigma \vecnorm*{\vect a}^2  + \left(\gamma - \left(\frac{1}{\varepsilon} - 1\right) \kappa_{\vect h}^2 \right) \vecnorm{\vect b}^2.
\end{align}
Letting $\varepsilon$ be such that the coefficient of $\vecnorm{\vect b}^2$ is $\gamma/2$,
we finally obtain
\begin{equation}
    \label{eq:final_step_bound_covariance_QPmu}
    \begin{pmatrix}
        \vect a \\ \vect b
    \end{pmatrix}^\t
    \mathcal C(\op Q \op P \mu)
    \begin{pmatrix}
        \vect a \\ \vect b
    \end{pmatrix}
    \geq  \frac{\gamma\sigma}{2 \kappa_{\vect h}^2+\gamma}
    \vecnorm{\vect a}^2 + \frac{\gamma}{2} \vecnorm{\vect b}^2,
\end{equation}
which concludes the proof.
\end{proof}

\begin{remark}
    A bound sharper than~\eqref{eq:final_step_bound_covariance_QPmu} can be obtained by letting~$\varepsilon$ be such that the coefficients of~$\vecnorm{\vect a}^2$ and~$\vecnorm{\vect b}^2$ are equal in~\eqref{eq:intermediate_moment},
    but this is not necessary for our purposes.
\end{remark}


\begin{lemma}
    \label{lemma:moment_bound}
    For $\mu_1, \mu_2 \in \mathcal P(\real^n)$ with finite second moments,
    it holds that
    \begin{align*}
        \bigl\lvert \mathcal M(\mu_1) - \mathcal M(\mu_2) \bigr\rvert &\leq \frac{1}{2} d_g(\mu_1, \mu_2), \\
        \bigl\lVert \mathcal C(\mu_1) - \mathcal C(\mu_2) \bigr\rVert &\leq \left(1 + \frac{1}{2} \vecnorm{\mathcal M(\mu_1) + \mathcal M(\mu_2)}\right) \, d_g(\mu_1, \mu_2).
    \end{align*}
\end{lemma}
\begin{proof}
    Let $\vect m_i = \mathcal M(\mu_i)$ and $\mat \Sigma_i = \mathcal C(\mu_i)$ for $i = 1, 2$.
    Notice that $|2a^\t u|\le g(u)$ if $|a|=1$, so
    \begin{align}
        \label{eq:bound_diff_means}
        \vecnorm{\vect m_1 - \vect m_2}
        = \sup_{\vecnorm{\vect a} = 1} \abs*{ \vect a^\t \bigl(\vect m_1 - \vect m_2\bigr) }
        = \sup_{\vecnorm{\vect a} = 1} \abs*{\mu_1\bigl[\vect a^\t u\bigr] - \mu_2\bigl[\vect a^\t u\bigr]}
        \leq \frac{1}{2} d_g(\mu_1, \mu_2),
    \end{align}
    where the supremum is over the unit sphere in $\real^n$,
    centered at the origin and in the Euclidean distance.
    Similarly
    \begin{align}
        \notag
        \norm*{\mat \Sigma_1 - \mat \Sigma_2}
        &= \sup_{\vecnorm{\vect a} = 1} \abs*{\vect a^\t \mat \Sigma_1 \vect a - \vect a^\t \mat \Sigma_2 \vect a} \\
        \notag
        &\leq \sup_{\vecnorm{\vect a} = 1}\left\{ \abs*{\mu_1\bigl[\lvert \vect a^\t u \rvert^2\bigr] - \mu_2\bigl[\lvert \vect a^\t u \rvert^2\bigr]}
        + \abs*{\mu_1\bigl[\vect a^\t u\bigr]^2 - \mu_2\bigl[\vect a^\t u\bigr]^2}\right\} \\
        \notag
        &\leq  d_g(\mu_1, \mu_2) +  \sup_{\vecnorm{\vect a} = 1}\left|\mu_1\bigl[\vect a^\t u\bigr] + \mu_2\bigl[\vect a^\t u\bigr]\right|\left|\mu_1\bigl[\vect a^\t u\bigr] - \mu_2\bigl[\vect a^\t u\bigr]\right|\\
        \label{eq:bound_diff_cov}
        &\leq \left(1 + \frac{1}{2} \vecnorm{\vect m_1 + \vect m_2}\right) \, d_g(\mu_1, \mu_2),
    \end{align}
    which concludes the proof.
\end{proof}


\subsection{Action of \texorpdfstring{$\op T_j$}{} on Gaussians}
\label{sub:mean_field_is_conditioning_for_gaussians}
\begin{lemma}
    [$\op B_j \op G = \op T_j \op G$]
    \label{lemma:mean_field_map}
    Fix $y^\dagger_{j+1}\in\real^K$. Let $\pi$ be a Gaussian measure over $\real^d \times \real^K$ with mean and covariance given by
    \[
        \vect m =
        \begin{pmatrix}
            \vect m_u \\
            \vect m_y
        \end{pmatrix},
        \qquad
        \mat S =
        \begin{pmatrix}
            \mat S_{uu} & \mat S_{uy} \\
            \mat S_{uy}^\t & \mat S_{yy}
        \end{pmatrix}.
    \]
        Recall the pushforward of $\mathscr T$ defined in~\eqref{eq:Tmap}.
        Then the probability measure $\op B_j \pi$ defined in~\eqref{eq:definition_B} coincides with the probability measure $\op T_j\pi =\mathscr T (\placeholder, \placeholder;\pi, y^\dagger_{j+1})_\sharp \pi$.
\end{lemma}
\begin{proof}
    For conciseness, we denote $y^\dagger = y^\dagger_{j+1}$.
    Using the well known formula for the conditional distribution of a normal random variable,
    we have that
    \begin{equation}
        \label{eq:conditional_distribution}
        \op B_j \pi =
        \normal \left( \vect m_u + S_{uy} S_{yy}^{-1} (y^\dagger - \vect m_y), S_{uu} -  S_{uy} S_{yy}^{-1} S_{uy}^\t\right).
    \end{equation}
    Since $\op T_j$ is the pushforward under an affine map,
    it maps Gaussian distributions in $\mathcal P(\real^d \times \real^K)$ to Gaussian distributions in $\mathcal P(\real^d)$,
    and so is sufficient to check that the first and second moments of~$\op B_j \pi$ and~$\op T_j \pi$ coincide.
    By definition of~$\op T_j$ in~\eqref{eq:Tmap},
    we have that $U + S_{uy} S_{yy}^{-1} (y^\dagger - Y) \sim \op T_j \pi$ if~$(U, Y) \sim \pi$.
    It follows immediately that the mean under $\op T_j \pi$ coincides with that under the conditional distribution~\eqref{eq:conditional_distribution}.
    Employing the expression for the mean under $\op T_j \pi$,
    we then obtain that the covariance under $\op T_j \pi$ is given by
    \[
        \mathcal C (\op T_j \pi) = \expect_{(U, Y) \sim \pi} \biggl[ \Bigl(U - \vect m_u + S_{uy} S_{yy}^{-1} \bigl(\vect m_y - Y\bigr)\Bigr) \otimes \Bigl(U - \vect m_u + S_{uy} S_{yy}^{-1} \bigl(\vect m_y - Y\bigr)\Bigr) \biggr].
    \]
    Developing this expression, we obtain
    \begin{align*}
        \mathcal C (\op T_j \pi)
        &= \expect_{(U, Y) \sim \pi}
        \biggl[ (U - \vect m_u) (U - \vect m_u)^\t
            + (U - \vect m_u) \bigl(\vect m_y - Y\bigr)^\t S_{yy}^{-1}  S_{uy}^\t \\
        &\quad \qquad \qquad + S_{uy} S_{yy}^{-1} \bigl(\vect m_y - Y\bigr)  (U - \vect m_u)^\t
        + S_{uy} S_{yy}^{-1} \bigl(\vect m_y - Y\bigr)  \bigl(\vect m_y - Y\bigr)^\t S_{yy}^{-1}  S_{uy}^\t \biggr]\\
        &= S_{uu}  - S_{uy} S_{yy}^{-1} S_{uy}^\t - S_{uy} S_{yy}^{-1} S_{uy}^\t + S_{uy} S_{yy}^{-1} S_{uy}^\t
        = S_{uu} - S_{uy} S_{yy}^{-1} S_{uy}^\t,
    \end{align*}
    which indeed coincides with the covariance of $\op B_j \pi$ in~\eqref{eq:conditional_distribution}.
\end{proof}

\begin{lemma}
    \label{lem:newform}
    The maps\eqref{eq:gaussian_projection_filter} and \eqref{eq:gaussian_projection_filter2} are equivalent.
\end{lemma}

\begin{proof}
    The equivalence follows from~\cref{lemma:mean_field_map} and the operator equality $\op T_j\op G=\op G\op T_j$,
    with both sides viewed as operators from $\mathcal P(\real^d \times \real^K)$ to $\mathcal P(\real^d)$.
    Since $\op T_j$ maps Gaussians to Gaussians,
    the image of both operators is contained in the set $\mathcal G(\real^d)$ of Gaussian distributions.
    It is therefore sufficient to check that for any~$\pi \in \mathcal P(\real^d \times \real^K)$,
    the probability measures $\op T_j\op G \pi$ and $\op G\op T_j \pi$ have the same first and second moments.
    We saw in the proof of \cref{lemma:mean_field_map} that the first and second moments of $\op T_j p$,
    for any~$p \in \mathcal P(\real^d \times \real^K)$, depend only on the first and second moments of $p$.
    Therefore the first and second moments of $\op T_j \pi$ and $\op T_j \op G \pi$ coincide,
    since the operator $\op G$ leaves the first and second moments invariant,
    and the conclusion follows.
\end{proof}

\subsection{Stability Results}
\label{sub:stability_results}

\begin{lemma}
    [The map $\op P$ is globally Lischitz]
    \label{lemma:lipschitz_p}
    Under~\cref{assumption:ensemble_kalman},
    it holds for all $\mu \in \mathcal P(\real^d)$ that
    \[
        d_g(\op P \mu, \op P \nu)
        \leq \Bigl(1+\kappa_{\vect \Psi}^2 + \trace(\mat \Sigma)\Bigr) \, d_g (\mu, \nu).
    \]
\end{lemma}
\begin{proof}
    By definition~\eqref{eq:Markov_kernel} of~$\op P$,
    it holds that
    \[
        \op P \mu(u) = \int_{\real^d} \frac{\exp \left( - \frac{1}{2} \vecnorm*{u - \vect \Psi(v)}_{\mat \Sigma}^2 \right)}{\sqrt{(2\pi)^d \det \mat \Sigma}} \, \mu(\d v)
        =: \int_{\real^d} p(v, u)  \, \mu(\d v).
    \]
    Take any $f\colon \real^d \to \real$ such that $f \leq g$, $g(u)=1+|u|^2$.
    Since $\vect \Psi$ is bounded by assumption,
    \begin{align*}
        b(v) := \int_{\real^d} f(u)  p(v, u) \, \d u
        &\leq \int_{\real^d} g(u)  p(v, u) \, \d u  \\
        &= 1+\left\lvert \vect \Psi(v) \right\rvert^2  + \trace(\mat \Sigma)
        \leq 1+\kappa_{\vect \Psi}^2 + \trace(\mat \Sigma).
    \end{align*}
    Therefore,
    using Fubini's theorem,
    we have
    \begin{align*}
        \Bigl\lvert \op P \mu [f] - \op P \nu [f] \Bigr\rvert
        &= \biggl\lvert \int_{\real^d} \left( \int_{\real^d} f(u)  p(v, u) \, \d u \right)
        \bigl(\mu(\d v) - \nu(\d v)\bigr)  \biggr\rvert \\
        &= \Bigl\lvert \mu [b] - \nu [b] \Bigr\rvert
        \leq \Bigl(1+\kappa_{\vect \Psi}^2 + \trace(\mat \Sigma)\Bigr) d_g(\mu, \nu),
    \end{align*}
    which concludes the proof.
\end{proof}

\begin{lemma}
    [The map $\op Q$ is globally Lipschitz]
    \label{lemma:lipschitz_Q}
    Under~\cref{assumption:ensemble_kalman},
    it holds for any $\mu, \nu \in \mathcal P(\real^d)$ that
    \begin{equation}
        \label{eq:lipschitz_Q}
        d_{g}(\op Q \mu, \op Q \nu)
        \leq \Bigl(1 + \kappa_{\vect h}^2 + \trace(\mat \Gamma) \Bigr) d_{g}(\mu, \nu).
    \end{equation}
\end{lemma}
\begin{proof}
    Let us take $f \colon \real^d \times \real^K \to \real$ such that $f \leq g$ and introduce the operator~$\Pi$ given by
    \[
        \Pi f(u) = \int_{\real^K} f(u, y) \, \normal\bigl(\vect h(u), \mat \Gamma\bigr) (\d y).
    \]
    Clearly $\Pi f(u) \leq \Pi g(u)$,
    and it holds that
    \begin{align*}
        \Pi g(u)
        &= \int_{\real^K} \left(1 + \vecnorm{u}^2 + \vecnorm{y}^2\right)  \, \normal\bigl(\vect h(u), \mat \Gamma\bigr) (\d y)
        = 1 + \vecnorm{u}^2 + \int_{\real^K} \vecnorm{y}^2 \normal\bigl(\vect h(u), \mat \Gamma\bigr) (\d y) \\
        &= 1 + \vecnorm{u}^2 + \bigl\lvert \vect h(u) \bigr\rvert^2 + \trace(\mat \Gamma)
        \leq \Bigl(1 + \kappa_{\vect h}^2 + \trace(\mat \Gamma) \Bigr) \left(1 + \vecnorm{u}^2\right).
    \end{align*}
    Therefore
    \[
        \Bigl\lvert \op Q \mu[f] - \op Q \nu[f] \Bigr\rvert
        = \Bigl\lvert \mu[\Pi f] - \nu[\Pi f] \Bigr\rvert
        \leq \Bigl(1 + \kappa_{\vect h}^2 + \trace(\mat \Gamma) \Bigr) d_{g}(\mu, \nu),
    \]
    and we obtain~\eqref{eq:lipschitz_Q}.
\end{proof}

\begin{lemma}
    \label{lemma:missing_lemma}
    Under~\cref{assumption:ensemble_kalman},
    there exists $C_{\op B} = C_{\op B}(\kappa_y, \kappa_{\vect \Psi}, \kappa_{\vect h}, \mat \Sigma, \mat \Gamma)$ such that
    for any probability measure $\mu \in \mathcal P(\real^d)$,
    it holds that
    \[
        \forall j \in \llbracket 0, J \rrbracket,
        \qquad
        d_{g}(\op B_j \op G \op Q \op P \mu, \op B_j \op Q \op P \mu)
        \leq C_{\op B} d_{g}(\op G \op Q \op P \mu, \op Q \op P \mu).
    \]
\end{lemma}
\begin{proof}
    For conciseness, we denote $y^\dagger = y^\dagger_{j+1}$.
    Let us introduce the $y$-marginal densities
    \[
        \alpha_\mu(y):= \int_{\real^d} \op G \op Q \op P\mu(u, y) \, \d u\,,
        \qquad
        \beta_\mu(y):=\int_{\real^d} \op Q \op P\mu(u, y) \, \d u\,.
    \]
    Then
    \begin{align}
        \notag
        &d_{g}(\op B_j \op G \op Q \op P \mu, \op B_j \op Q \op P \mu)
        = \int_{\real^d} \left(1+|u|^2\right) \left| \frac{\op G \op Q \op P\mu(u,y^\dagger)}{\alpha_\mu(y^\dagger)} - \frac{\op Q \op P\mu(u,y^\dagger)}{\beta_\mu(y^\dagger)}\right| \, \d u\\
        \notag
        &\qquad \qquad \le \frac{1}{\alpha_\mu(y^\dagger)}\int_{\real^d} \left(1+|u|^2\right) \bigl\lvert \op G \op Q \op P\mu(u,y^\dagger)-\op Q \op P\mu(u,y^\dagger)\bigr\rvert \, \d u \\
        \label{eq:main_equation}
        &\qquad \qquad \qquad + \left| \frac{\alpha_\mu(y^\dagger)-\beta_\mu(y^\dagger)}{\alpha_\mu(y^\dagger) \beta_\mu(y^\dagger)}\right|\,\int_{\real^d} \left(1+|u|^2\right) \op Q \op P\mu(u,y^\dagger) \, \d u.
    \end{align}

    \paragraph{Step 1: bounding $\alpha_{\mu}(y^\dagger)$ and $\beta_{\mu}(y^\dagger)$ from below}
    The marginal distribution $\alpha_{\mu}(\placeholder)$ is Gaussian with covariance matrix
    \begin{equation}
        \label{eq:covariance_matrix_second_gaussian}
        \mat \Gamma + \op P \mu\bigl[\vect h(\placeholder) \otimes \vect h(\placeholder)\bigr] - \op P \mu\bigl[\vect h(\placeholder)\bigr] \otimes \op P \mu\bigl[\vect h(\placeholder)\bigr],
    \end{equation}
    which is bounded from below by $\mat \Gamma$ and from above by $\mat \Gamma + \kappa_{\vect h}^2 \mat I_{K}$,
    in view of \cref{assumption:ensemble_kalman}.
    (We again use the fact that $\vect a \vect a^\t \preccurlyeq (\vect a^\t \vect a) \mat I_{d}$ for any vector $\vect a \in \real^d$, by the Cauchy-Schwarz inequality.)
    Assumption~\ref{assumption:ensemble_kalman} also implies that the mean of $\alpha_{\mu}$ is bounded from above in norm by~$\kappa_{\vect h}$. Therefore, it holds that
    \begin{equation}
        \label{eq:lower_bound_alpha}
        \forall y \in \real^K, \qquad
        \alpha_\mu(y)
        \geq \frac{\exp\Bigl(- \frac{1}{2}  (|y| + \kappa_{\vect h})^2  \norm*{\mat \Gamma^{-1}}\Bigr)}{\sqrt{(2\pi)^K \det (\mat \Gamma + \kappa_{\vect h}^2 \mat I_K)}}.
    \end{equation}
    The function $\beta_{\mu}$ can be bounded from below independently of $\mu$ in a similar manner.
    Indeed, it holds under \cref{assumption:ensemble_kalman} that for all $y \in \real^K$,
    \begin{align}
        \nonumber
        \beta_{\mu}(y)
        &= \int_{\real^d} \op Q \op P \mu(u, y) \, \d u
        = \int_{\real^d} \frac{\exp\left(- \frac{1}{2} \bigl(y - \vect h(u)\bigr)^\t \mat \Gamma^{-1}  \bigl(y - \vect h(u)\bigr) \right)}{\sqrt{(2\pi)^K \det(\mat\Gamma)}} \op P\mu(u) \, \d u \\
        \label{eq:lower_bound_beta}
        &\geq \frac{\exp\Bigl(- \frac{1}{2}  (|y| + \kappa_{\vect h})^2  \norm*{\mat \Gamma^{-1}}\Bigr)}{\sqrt{(2\pi)^K \det (\mat \Gamma)}}.
    \end{align}

    \paragraph{Step 2: bounding the first term in~\eqref{eq:main_equation}}
    For fixed $u \in \real^d$, the functions $y \mapsto \op Q \op P \mu(u, y)$ and~$y \mapsto \op G \op Q \op P \mu(u, y)$ are Gaussians up to constant factors.
    The covariance matrix of the former is~$\mat \Gamma$ and,
    using the formula for the covariance of the conditional distribution of a Gaussian,
    we calculate that the covariance of the latter is given by
    \begin{equation}
        \label{eq:conditioned_covariance}
        \mathcal C^{yy}(\op Q \op P \mu) - \mathcal C^{yu} (\op Q \op P \mu) \mathcal C^{uu}(\op Q \op P \mu)^{-1}\mathcal C^{uy}(\op Q \op P \mu)
    \end{equation}
    Since $\mathcal C^{yu} (\op Q \op P \mu) \mathcal C^{uu}(\op Q \op P \mu)^{-1}\mathcal C^{uy}(\op Q \op P \mu)$ is positive semi-definite,
    it follows from~\cref{lemma:bound_on_QPmu} that the matrix~\eqref{eq:conditioned_covariance} is bounded from above by~$2\kappa_{\vect h}^2 \mat I_K + \mat \Gamma$.
    Then, using the same notation as in~\eqref{eq:decomposition_covariance},
    we obtain that
    \begin{multline}
        \mathcal C^{yy}(\op Q \op P \mu) - \mathcal C^{yu} (\op Q \op P \mu) \mathcal C^{uu}(\op Q \op P \mu)^{-1}\mathcal C^{uy}(\op Q \op P \mu)
         \\ =
        \mat \Gamma + \Bigl( \mathcal C^{yy}(\phi_\sharp \op P  \mu) - \mathcal C^{yu} (\phi_\sharp \op P  \mu) \mathcal C^{uu}(\phi_\sharp \op P  \mu)^{-1}\mathcal C^{uy}(\phi_\sharp \op P  \mu) \Bigr)
        \succcurlyeq \mat \Gamma,
    \end{multline}
    where the inequality holds because,
    being the Schur complement of the block $\mathcal C^{uu} (\phi_\sharp \op P  \mu)$ of the matrix $\mathcal C(\phi_\sharp \op P  \mu)$,
    the bracketed term is positive semi-definite.
    Therefore, the matrix~\eqref{eq:conditioned_covariance} is bounded from below by $\mat \Gamma$,
    and so the integral in the first term of~\eqref{eq:main_equation} may be bounded from above by using~\cref{lemma:gaussians_relation_1_and_inf_norms} in~\cref{app:A2} with parameter $\alpha=\alpha(\kappa_{\vect h}, \mat \Gamma)$,
    which gives
    \begin{align}
        \notag
        &\int_{\real^d} \left(1+|u|^2\right) \left| \op G \op Q \op P\mu(u,y)-\op Q \op P\mu(u,y)\right|\d u \\
        \notag
        &\qquad \leq C \int_{\real^K}  \int_{\real^d} \left(1+|u|^2\right) \bigl\lvert \op G \op Q \op P\mu(u,y)-\op Q \op P\mu(u,y)\bigr\rvert \, \d u \, \d y\\
        \notag
        &\qquad \leq C \int_{\real^K}  \int_{\real^d} \left(1+|u|^2+|y|^2\right) \bigl\lvert \op G \op Q \op P\mu(u,y)-\op Q \op P\mu(u,y)\bigr\rvert \, \d u \, \d y\\
        \label{eq:bound_first_term}
        &\qquad = C d_{g} (\op G \op Q \op P\mu, \op Q \op P \mu).
    \end{align}
    for an appropriate constant $C$ depending only on $\mat \Gamma$ and~$\kappa_{\vect h}$.

    \paragraph{Step 3: bounding the second term in~\eqref{eq:main_equation}}
    By~\eqref{eq:bound_first_term}, we have
    \begin{align*}
        \bigl\lvert \alpha_\mu(y)-\beta_\mu(y) \bigr\rvert& = \left| \int_{\real^d} \op G \op Q \op P\mu(u, y) - \op Q \op P\mu(u, y) \, \d u\right| \\
                                &\leq \int_{\real^d} \left(1+|u|^2\right) \bigl\lvert  \op G \op Q \op P\mu(u,y)-\op Q \op P\mu(u,y)\bigr \rvert\,\d u
         \leq C d_{g} (\op G \op Q \op P\mu, \op Q \op P \mu).
    \end{align*}
    On the other hand,
    since $y \mapsto \op Q \op P \mu(u, y) / \op P \mu(u)$ is a Gaussian density with covariance $\mat \Gamma$,
    which is bounded uniformly from above by $\bigl((2\pi)^K \det (\mat \Gamma)\bigr)^{-1/2}$,
    it holds that
    \begin{align*}
        \int_{\real^d} \left(1+|u|^2\right) \op Q \op P\mu(u,y) \, \d u
        &= \beta_{\mu}(y) + \int_{\real^d} |u|^2 \op Q \op P\mu(u,y) \, \d u \\
        &\leq \beta_{\mu}(y) + \int_{\real^d}  \frac{|u|^2 \op P \mu(u)}{\sqrt{(2\pi)^K \det(\Gamma})} \, \d u \\
        &= \beta_{\mu}(y) + \frac{\trace\bigl(\mathcal C(\op P\mu)\bigr) + \vecnorm*{\mathcal M(\op P\mu)}^2}{\sqrt{(2\pi)^K \det (\mat \Gamma})}.
    \end{align*}
    By \cref{lemma:bound_on_Pmu},
    the first and second moments of $\op P \mu$ are bounded from above by a constant depending only on~$\kappa_{\vect \Psi}$ and $\mat\Sigma$.

    \paragraph{Step 4: concluding the proof}
    Combining the above inequalities, we deduce that
    \begin{align*}
        &d_{g}(\op B_j \op G \op Q \op P \mu, \op B_j \op Q \op P \mu)
        \le \frac{C(\kappa_{\vect \Psi}, \kappa_{\vect h}, \mat \Sigma, \mat \Gamma)}{\alpha_{\mu}(y^\dagger_{j+1})} \left(1 + \frac{1}{\beta_\mu(y^\dagger_{j+1})} \right)  d_{g} (\op G \op Q \op P\mu, \op Q \op P \mu).
    \end{align*}
    Using~\eqref{eq:lower_bound_alpha} and~\eqref{eq:lower_bound_beta} together with the uniform bound $\kappa_y$ on the data (\cref{assumption:ensemble_kalman}) gives the conclusion.
\end{proof}

To conclude this section,
we show a stability result for $\op T_j$.
\begin{lemma}
    [Stability result for the mean field map $\op T_j$]
    \label{lemma:lipschitz_mean_field_affine}
    Suppose that \cref{assumption:ensemble_kalman} is satisfied and that $|h|_{C^{0,1}} \leq \ell_{\vect h} < \infty$.
    Then, for all $R \geq 1$, there is $L_{\op T} = L_{\op T}(R, \kappa_y, \kappa_{\vect \Psi}, \kappa_{\vect h}, \ell_{\vect h}, \mat \Sigma, \mat \Gamma)$ such that for all $\pi \in \mathcal P_R(\real^d \times \real^K)$ and~$\mu \in \mathcal P(\real^d)$,
    it holds that
    \[
        \forall j \in \range{1}{J}, \qquad
        d_{g}(\op T_j \pi, \op T_j p)
        \leq L_{\op T} \, d_{g}(\pi, p),
        \qquad p := \op Q \op P \mu.
    \]
\end{lemma}

\begin{proof}
    By \cref{lemma:bound_on_QPmu},
    the probability measure $\op Q \op P \mu$ belongs to $\mathcal P_{\widetilde R}(\real^d \times \real^K)$ for some $\widetilde R \ge 1$.
    Let us introduce
    \[
        r = \max \Bigl\{R, \widetilde R, \kappa_y \Bigr\}
    \]
    and denote by $\mathscr T^{\pi}$ and $\mathscr T^{p}$ the affine maps corresponding to evaluation
    of covariance information at $\pi$ and $p= \op Q \op P \mu$. Specifically,
    \begin{alignat*}{2}
        \mathscr T^{\pi}(u, y) &= u + \mat A_{\pi} (y^\dagger - y), \qquad & \mat A_{\pi} := \mat K_{uy} \mat K_{yy}^{-1}, \\
        \mathscr T^{p}(u, y) &= u + \mat A_{p} (y^\dagger - y), \qquad & \mat A_{p} := \mat S_{uy} \mat S_{yy}^{-1},
    \end{alignat*}
    where $\mat K = \mathcal C(\pi)$, $\mat S = \mathcal C(p)$ and $y^\dagger = y^\dagger_{j+1}$.
    By the triangle inequality, we have
    \begin{equation}
        \label{eq:main_equation_filtering}
        d_{g}(\op T_j \pi, \op T_j p)
        \leq  d_{g}(\mathscr T^{\pi}_{\sharp} \pi, \mathscr T^{\pi}_{\sharp} p) + d_{g} (\mathscr T^{\pi}_{\sharp} p, \mathscr T^{p}_{\sharp} p).
    \end{equation}
    Before separately bounding each term on the right-hand side,
    we obtain simple inequalities that will be helpful later in the proof.
    It holds that
    \begin{equation}
        \label{eq:lipschitz_mean_field_auxiliary_1}
        \matnorm{\mat A_{\pi}} \leq \matnorm{\mat K_{uy}} \matnorm{\mat K_{yy}^{-1}}
        \leq \matnorm{\mat K} \matnorm{\mat K^{-1}} \leq r^4,
    \end{equation}
    and similarly for~$\mat A_p$.
    Here we used that the matrix norm (induced by the Euclidean vector norm)
    of any submatrix is bounded from above by the norm of the original matrix.
    Using this bound again, we obtain, assuming $r>1$ without any loss of generality,
    \begin{align}
        \nonumber
        \matnorm{\mat A_{\pi} - \mat A_p} = \matnorm{\mat K_{uy} \mat K_{yy}^{-1} - \mat S_{uy} \mat S_{yy}^{-1}}
        &\leq \matnorm{(\mat K_{uy} - \mat S_{uy}) \mat K_{yy}^{-1}} + \matnorm{\mat S_{uy} \left(\mat K_{yy}^{-1} - \mat S_{yy}^{-1}\right)} \\
        \nonumber
        &\leq r^2 \Bigl( \matnorm{\mat K_{uy} - \mat S_{uy}} + \matnorm{\mat K_{yy}^{-1} - \mat S_{yy}^{-1}} \Bigr)\\
        \nonumber
        &= r^2 \Bigl( \matnorm{\mat K_{uy} - \mat S_{uy}} + \matnorm{\mat K_{yy}^{-1} (\mat S_{yy} - \mat K_{yy}) \mat S_{yy}^{-1}} \Bigr) \\
        \nonumber
        &\leq r^2 \Bigl( \matnorm{\mat K_{uy} - \mat S_{uy}} + \matnorm{\mat K_{yy}^{-1}} \matnorm{\mat S_{yy} - \mat K_{yy}} \matnorm{\mat S_{yy}^{-1}} \Bigr) \\
        \label{eq:lipschitz_mean_field_auxiliary_2}
        &\leq 2r^6  \matnorm{\mat K - \mat S} \leq 2r^6 (1 + 2r) \, d_{g}(\pi, p),
    \end{align}
    where we used in the last line the inequality~\eqref{eq:bound_diff_cov} from~\cref{lemma:moment_bound}.

    \paragraph{Bounding the first term in~\eqref{eq:main_equation_filtering}}
    Let $f\colon \real^d \to \real$ denote a function satisfying $\abs{f} \leq g$.
    It follows from the definition of the pushforward measure that
    \[
        \bigl\lvert \mathscr T^{\pi}_{\sharp}\pi[f] - \mathscr T^{\pi}_{\sharp} p[f] \bigr\rvert
        = \bigl\lvert \pi[f\circ \mathscr T^{\pi}] - p[f\circ \mathscr T^{\pi}] \bigr\rvert.
    \]
    For all $(u, y) \in \real^d \times \real^K$, it holds that
    \begin{align*}
        \bigl\lvert f\circ \mathscr T^{\pi}(u, y) \bigr\rvert &= \abs*{f\Bigl(u + \mat A_{\pi} \bigl(y^\dagger - y\bigr)\Bigr)} \leq g\Bigl(u + \mat A_{\pi} \bigl(y^\dagger - y\bigr)\Bigr) \\
                                   &= 1 + \bigl\lvert u + \mat A_{\pi} (y^\dagger - y) \bigr\rvert^2 \leq 1 + 3\vecnorm{u}^2 +  3 \lvert \mat A_{\pi} y^\dagger \rvert^2 + 3 \vecnorm{\mat A_{\pi} y}^2 \\
                                   &\leq 3 \left(1 + \lvert \mat A_{\pi} y^\dagger \rvert^2\right) \max \left\{ 1, \matnorm{A_{\pi}}^2\right\} g(u, y).
    \end{align*}
    Therefore, using~\eqref{eq:lipschitz_mean_field_auxiliary_1} we deduce that
    \begin{equation}
        \label{eq:stability_T_first_term}
        d_{g}(\mathscr T^{\pi}_{\sharp} \pi, \mathscr T^{\pi}_{\sharp} p)
        \leq 3 \left(1 + r^{10}\right) r^8 \, d_{g}(\pi, p).
    \end{equation}

    \paragraph{Bounding the second term in~\eqref{eq:main_equation_filtering}}
    Let $f\colon \real^d \to \real$ again denote a function satisfying the inequality $\abs{f} \leq g$, with $g(u) = 1 + \vecnorm{u}^2$,
    It holds that
    \[
        \Bigl\lvert \mathscr T^{\pi}_{\sharp}p[f] - \mathscr T^{p}_{\sharp} p[f] \Bigr\rvert
        = \Bigl\lvert p[f\circ \mathscr T^{\pi}] - p[f\circ \mathscr T^{p}] \Bigr\rvert
        = \Bigl\lvert p[f\circ \mathscr T^{\pi} - f\circ \mathscr T^{p}] \Bigr\rvert.
    \]
    The right-hand side may be rewritten more explicitly as
    \begin{align*}
        \Bigl\lvert p[f\circ \mathscr T^{\pi} - f\circ \mathscr T^{p}] \Bigr\rvert
        &= \abs*{\int_{\real^K}\int_{\real^d} f\left(u + \mat A_{\pi}\bigl(y^\dagger - y\bigr)\right) - f\left(u + \mat A_{p}\bigl(y^\dagger - y\bigr)\right) \, p(u,y) \, \d u \, \d y}.
    \end{align*}
    We apply a change of variable in order rewrite the  integral in this
    identity as
    \begin{align}
        \nonumber
        &\int_{\real^K} \int_{\real^d} \bigl(f(u + \mat A_{\pi}  z ) - f(u + \mat A_{p}  z )\bigr) \,  p (u, y^\dagger -  z) \, \d u \, \d z \\
        \label{eq:filtering_second_term}
        \qquad &= \int_{\real^K} \int_{\real^d} f(v) \Bigl( p (v - \mat A_{\pi} z, y^\dagger -  z) -  p (v - \mat A_{p} z, y^\dagger -  z) \Bigr) \d v \, \d z.
    \end{align}
    It follows from the technical~\cref{lemma:lipschitz_density} proved in~\cref{app:A2} that
    \begin{align*}
        &\Bigl\lvert  p (v - \mat A_{\pi} z, y^\dagger -  z) -  p (v - \mat A_{p} z, y^\dagger -  z) \Bigr \rvert \\
        &\qquad \leq  C\lvert \mat A_{\pi} z - \mat A_{p} z \rvert \exp \left(- \frac{1}{4} \left(\vecnorm*{y^\dagger - z}_{\mat \Gamma}^2 + \min \left\{ \vecnorm*{v - \mat A_{\pi} z}_{K}^2, \vecnorm*{v - \mat A_{p} z}_{\mat K}^2 \right\}  \right) \right).
    \end{align*}
    Using this inequality,
    we first bound the inner integral in~\eqref{eq:filtering_second_term} for fixed $z \in \real^K$.
    Keeping only the terms that depend on~$v$, we obtain that
    \begin{align*}
        &\int_{\real^d} \lvert f(v) \rvert \exp \left(- \frac{1}{4} \min \left\{ \vecnorm*{v - \mat A_{\pi} z}_{\mat K}^2, \vecnorm*{v - \mat A_{p} z}_{\mat K}^2 \right\}   \right) \, \d v \\
        &\qquad =
         \int_{\real^d} \lvert f(v) \rvert \max \left\{ \exp \left(- \frac{1}{4} \vecnorm*{v - \mat A_{\pi} z}_{\mat K}^2 \right),  \exp \left(- \frac{1}{4} \vecnorm*{v - \mat A_{p} z}_{\mat K}^2   \right)  \right\} \, \d v \\
        &\qquad \leq
        \int_{\real^d} \lvert f(v) \rvert  \exp \left(- \frac{1}{4} \vecnorm*{v - \mat A_{\pi} z}_{\mat K}^2   \right) \d v + \int_{\real^d} \lvert f(v) \rvert  \exp \left(- \frac{1}{4} \vecnorm*{v - \mat A_{p} z}_{\mat K}^2   \right) \d v .
    \end{align*}
    Since $\lvert f(v) \rvert \leq 1 + \vecnorm{v}^2$,
    it is clear that this expression is bounded from above by
    \[
        C \Bigl(1 + \lvert \mat A_{\pi} z \rvert^2 + \lvert \mat A_{p} z \rvert^2  \Bigr) \leq C r^8 \left(1 + \vecnorm{z}^2 \right),
    \]
    where we used~\eqref{eq:lipschitz_mean_field_auxiliary_1} in the last inequality.
    The remaining integral in the~$z$ direction can be bounded similarly:
    \begin{align}
        \nonumber
        d_{g} (\mathscr T^{\pi}_{\sharp} p, \mathscr T^{p}_{\sharp} p)
        &\leq C r^8 \int_{\real^K}  \left(1 + \vecnorm{z}^2 \right) \lvert \mat A_{\pi} z - \mat A_{p} z \rvert \exp \left(- \frac{1}{4} \left(\vecnorm*{y^\dagger - z}_{\mat \Gamma}^2 \right) \right)  \, \d z \\
        \label{eq:stability_T_second_term}
        &\leq C r^{11} \matnorm{\mat A_{\pi} - \mat A_{p}} \leq C r^{18} d_{g}(\pi, p),
    \end{align}
    where we used the bound~\eqref{eq:lipschitz_mean_field_auxiliary_2} in the last inequality.
    Combining~\eqref{eq:stability_T_first_term} and \eqref{eq:stability_T_second_term},
    we deduce the statement.
\end{proof}

\subsection{Additional Stability Results for the Gaussian Projected Filter}
\label{sub:additional_stability_results}
Next, we show the local Lipschitz continuity of the Gaussian projection map~$\op G$.
To this end, we begin by proving the following result.
(The proof uses an approach similar to that employed in~\cite{2018arXiv181008693D} for obtaining an upper bound on the distance between Gaussians in the usual total variation metric.)

\begin{lemma}
    \label{lemma:dg_gaussians}
    Let $\mu_1 = \normal(\vect m_1, \mat \Sigma_1)$ and $\mu_2 = \normal(\vect m_2, \mat \Sigma_2)$ with $\mat \Sigma_1, \mat \Sigma_2 \in \real^{n\times n}$ symmetric and positive definite.
    It holds that
    \begin{equation}
        \label{eq:dg_gaussians}
        d_g\bigl(\mu_1, \mu_2\bigr)
        \leq \sqrt{\mu_1\left[g^2\right] + \mu_2\left[g^2\right]}
        \left(3\norm*{\mat \Sigma_2^{-1} \mat \Sigma_1 - I_n}_{\rm F} + \vecnorm{\vect m_1 - \vect m_2}_{\mat \Sigma_2}\right),
    \end{equation}
    where $\norm{\placeholder}_{\rm F}$ is the Frobenius matrix norm.
\end{lemma}

\begin{proof}
    We denote by $(\lambda_i)_{1\leq i\leq n}$ the eigenvalues of~$\mat \Sigma_1^{-1} \mat \Sigma_2$.
    Being the product of two symmetric positive definite matrices,
    $\mat \Sigma_1^{-1} \mat \Sigma_2$ is real diagonalizable with real and positive eigenvalues.
    We note that
    \begin{align*}
        d_g(\mu_1, \mu_2) \leq \mu_1[g] + \mu_2[g] \leq \sqrt{\mu_1[g^2]} + \sqrt{\mu_2[g^2]}
        \leq \sqrt{2} \sqrt{\mu_1[g^2] + \mu_2[g^2]}.
    \end{align*}
    Now assume $\max\{\lambda_j: 1 \leq j \leq n\} \geq 2$.
    In this case there is integer $k$ such that $\lambda_k \geq 2$ and so
    \[
        3\norm{\mat \Sigma_2^{-1} \mat \Sigma_1 - I_n}_{\rm F}
        = 3 \sqrt{ \sum_{i=1}^{n} \lvert \lambda_i^{-1} - 1 \rvert^2 }
        \geq 3  \abs*{\lambda_k^{-1} - 1}
        \geq \frac{3}{2} \geq \sqrt{2}.
    \]
    Combining the previous two inequalities proves~\eqref{eq:dg_gaussians} in the case where
    \[
        \max\{\lambda_j: 1 \leq j \leq n\} \geq 2.
    \]
    We assume from now on that $0 < \lambda_i < 2$ for all $i \in \range{1}{n}$.
    Employing the same reasoning  as in~\cite[Lemma 3.1]{MR2730330}, we can prove
    \begin{align}
        \label{eq:weighted_pinsker}
        d_g(\mu_1, \mu_2)^2
        \leq 2 \Bigl(\mu_1\bigl[g^2\bigr] + \mu_2\bigl[g^2\bigr]\Bigr) \, \kl{\mu_2}{\mu_1},
    \end{align}
    where $\kl{\mu_2}{\mu_1}$ is the Kullback--Leibler (KL) divergence of $\mu_2$ from $\mu_1$,
    given by
    \begin{equation}
        \label{eq:kl}
        \kl{\mu_1}{\mu_2} := \int_{\real^d} \frac{\d \mu_1}{\d \mu_2} \, \log \left(\frac{\d \mu_1}{\d \mu_2}\right)  \, \d \mu_2.
    \end{equation}
    The proof of the inequality~\eqref{eq:weighted_pinsker} is presented in~\cref{sub:generalized_pinsker} in~\cref{app:A2} for completeness.
    The KL divergence between two Gaussians has a closed expression, which we rewrite in terms of the function $f_1(x) := x^{-1} - 1+\log x$,
    \begin{align*}
        \kl{\mu_2}{\mu_1}
        &= \frac{1}{2} \left( \trace (\mat \Sigma_2^{-1} \mat \Sigma_1 - I_n) + (\vect m_1 - \vect m_2)^\t \mat \Sigma_2^{-1} (\vect m_1 - \vect m_2) - \log \det (\mat \Sigma_2^{-1} \mat \Sigma_1) \right)\\
        &= \frac{1}{2} \left(\sum_{i=1}^{n} (\lambda_i^{-1} - 1) + (\vect m_1 - \vect m_2)^\t \mat \Sigma_2^{-1} (\vect m_1 - \vect m_2) + \sum_{i=1}^{n} \log \lambda_i \right) \\
        &= \frac{1}{2} \left( \sum_{i=1}^{n} f_1(\lambda_i) +\vecnorm{\vect m_1 - \vect m_2}_{\mat \Sigma_2}^2\right) .
    \end{align*}
    The function $f_1$ is pointwise bounded from above by $f_2(x) := (1 - x^{-1})^2$ for~$x \in (0, 2)$.
    To see this consider the function ${\varphi}(x)=f_2(x)-f_1(x)$.
    It suffices to note that $\varphi(1) = 0$ and
    \[
        \varphi'(x)
        =  -\frac{1}{x} \left(1 -  \frac{1}{x} \right) \left(1 - \frac{2}{x} \right).
    \]
    Since $\varphi'(x) < 0$ for $x \in (0, 1)$ and $\varphi'(x) > 0$ for $x \in (1, 2)$,
    the desired upper bound on $f_1(x)$ by~$f_2(x)$ for $x\in (0,2)$ follows.
    Therefore, since
    \[
        \sum_{i=1}^{n} f_1(\lambda_i)
        \leq
        \sum_{i=1}^{n} f_2(\lambda_i)
        = \norm{\mat \Sigma_2^{-1} \mat \Sigma_1 - I_n}_{\rm F}^2,
    \]
    we have
    \begin{align*}
        \kl{\mu_2}{\mu_1}
        &\le \frac{1}{2} \Bigl( \norm{\mat \Sigma_2^{-1} \mat \Sigma_1 - I_n}_{\rm F}^2+\vecnorm{\vect m_1 - \vect m_2}_{\mat \Sigma_2}^2 \Bigr) \\
        &\le \frac{1}{2} \Bigl( \norm{\mat \Sigma_2^{-1} \mat \Sigma_1 - I_n}_{\rm F} +\vecnorm{\vect m_1 - \vect m_2}_{\mat \Sigma_2} \Bigr)^2.
    \end{align*}
    Consequently, we deduce from \eqref{eq:weighted_pinsker} that
    \begin{align*}
        d_g(\mu_1, \mu_2)
        &\leq
        \sqrt{\mu_1\bigl[g^2\bigr] + \mu_2\bigl[g^2\bigr]}
        \Bigl( \norm{\mat \Sigma_2^{-1} \mat \Sigma_1 - I_n}_{\rm F} + \vecnorm{\vect m_1 - \vect m_2}_{\mat \Sigma_2}  \Bigr),
    \end{align*}
    which concludes the proof.
\end{proof}

\begin{lemma}
    \label{lemma:lipschitz_G}
    For all $R \geq 1$,
    there is $L_{\op G} = L_{\op G}(R, n)$ such that for all $\mu_1, \mu_2 \in \mathcal P_R(\real^n)$,
    the following inequality holds:
    \begin{equation}
        \label{eq:lipschitz_G}
        d_g(\op G \mu_1, \op G \mu_2) \leq L_{\op G}(R, n)  \, d_g(\mu_1, \mu_2),
    \end{equation}
\end{lemma}

Before proving this lemma we make two remarks on the form of the result.

\begin{remark}
    \Cref{lemma:lipschitz_p,lemma:lipschitz_Q} on the Lipschitz continuity of $\op P$ and $\op Q$ do not depend on the specific choice of the function $g$ in \cref{definition:weighted_total_variation};
    they hold true also for the usual total variation distance.
    In contrast,
    in the proof of~\cref{lemma:lipschitz_G} the specific choice of~$g$ is used
    in order to control differences between moments by means of $d_g(\placeholder, \placeholder)$.
\end{remark}

\begin{remark}
    A simple example shows that the constant $L_{\op G}(R, n)$ in~\eqref{eq:lipschitz_G} is divergent in the limit~$R \to \infty$,
    indicating that the Gaussian projection operator~$\op G$ is not globally Lipschitz continuous.
    Specifically,
    take $\varepsilon \leq \frac{1}{2}$ and consider the probability distributions
    \[
        \mu_1 = \normal(0, \varepsilon), \qquad
        \mu_2 =  \varepsilon \delta_{-1} + (1-2 \varepsilon) \mu_1 +\varepsilon \delta_{1}.
    \]
    By definition of $d_g$, it follows that
    \begin{align}
        \notag
        d_g(\mu_1, \mu_2)
        &\le \varepsilon g(-1) + 2 \varepsilon \mu_1[g] + \varepsilon g(1) \\
        \label{eq:total_variation_no_projection}
        &= 2 \varepsilon + 2 \varepsilon \left(1 + \varepsilon\right) + 2 \varepsilon \leq 8 \varepsilon.
    \end{align}
    On the other hand, $\op G \mu_2 = \normal\bigl(0, 2 \varepsilon + (1 - 2 \varepsilon) \varepsilon\bigr)$,
    and the variance of this Gaussian is bounded from below by $3 \varepsilon - 2 \varepsilon^2 \geq 2 \varepsilon$.
    Consequently, it holds that
    \[
        d_g(\op G \mu_1, \op G \mu_2)
        \geq d_1(\op G \mu_1, \op G \mu_2)
        \geq d_1\bigl(\normal(0, \varepsilon), \normal(0, 2 \varepsilon)\bigr)
        = d_1\bigl(\normal(0, 1), \normal(0, 2)\bigr),
    \]
    where $d_1$ is the usual total variation metric.
    The right-hand side of this inequality does not depend on $\varepsilon$,
    and so
    \[
        \frac{d_g(\op G \mu_1,\op G \mu_2)}{d_g(\mu_1, \mu_2)} \xrightarrow[\varepsilon \to 0]{} \infty.
    \]
    This proves that $\op G$ is not globally Lipschitz.
\end{remark}

\begin{proof}[Proof of Lemma \ref{lemma:lipschitz_G}]
    Let $\vect m_i = \mathcal M(\mu_i)$ and $\mat \Sigma_i = \mathcal C(\mu_i)$ for $i = 1, 2$.
    By \cref{lemma:moment_bound},
    it holds that
    \[
        \vecnorm{\vect m_1 - \vect m_2}_{\mat \Sigma_2} \leq
        \frac{1}{2} \sqrt{\norm{\mat \Sigma_2^{-1}}} \, d_g(\mu_1, \mu_2).
    \]
    In addition, also using \cref{lemma:moment_bound}, and the facts that $\norm{A}_{\rm F} \leq \sqrt{n} \, \norm{A}$  and
    $\norm{A B} \leq \norm{A} \norm{B}$ for any matrices $A, B \in \real^{n\times n}$,
    \[
        \norm{\mat \Sigma_2^{-1} \mat \Sigma_1 - I_n}_{\rm F} \leq \sqrt{n}\norm{\mat \Sigma_2^{-1}} \left(1 + \frac{1}{2}\vecnorm{\vect m_1 + \vect m_2}\right) \, d_g(\mu_1, \mu_2).
    \]
    \Cref{lemma:dg_gaussians} then gives
    \begingroup
    \footnotesize
    \begin{align*}
        d_g(\op G \mu_1, \op G \mu_2)
        &\leq \sqrt{\op G \mu_1\left[g^2\right] + \op G \mu_2\left[g^2\right]}
        \left(3 \sqrt{n}\left(1 + \frac{1}{2} \vecnorm{\vect m_1 + \vect m_2}\right) \norm*{\mat \Sigma_2^{-1}} + \frac{1}{2}\sqrt{\norm*{\mat \Sigma_2^{-1}}} \right) d_g (\mu_1, \mu_2) \\
        &\leq \sqrt{\op G \mu_1[g^2] + \op G \mu_2[g^2]}
        \left(3 \sqrt{n}\left(1 + \frac{1}{2}\vecnorm{\vect m_1 + \vect m_2}\right) + 1 \right) \left(1 + \norm*{\mat \Sigma_2^{-1}}\right) d_g (\mu_1, \mu_2).
    \end{align*}
    \endgroup
    Clearly $\op G \mu_i \left[g^2\right] \leq C\bigl(1 + \vecnorm{\vect m_i}^4 + \norm{\mat \Sigma_i}^2\bigr)$ for our choice of $g$,
    for an appropriate constant~$C$,
    and so we conclude by use of the Young inequality that
    \[
        d_g(\op G \mu_1, \op G \mu_2)
        \leq C
        \left(1 + \vecnorm{\vect m_1}^3 + \vecnorm{\vect m_2}^3 + \norm*{\mat \Sigma_1}^{3/2} + \norm*{\mat \Sigma_2}^{3/2}\right)
        \bigl(1 + \norm*{\mat \Sigma_2^{-1}}\bigr) \,
        d_g (\mu_1, \mu_2),
    \]
     for another constant $C$ depending only on $n$; this completes the proof.
     \end{proof}

     We end the subsection by showing local Lipschitz continuity of the conditioning operator~$\op B_j$ over the set~$\mathcal G(\real^d \times \real^K)$.

\begin{lemma}
    \label{lemma:lipschitz_l2}
    For all $R \geq 1$,
    there exists a constant $L_{\op B} = L_{\op B}(R, \kappa_y)$ such that for all Gaussian probability measures $\pi, p \in \mathcal G_R(\real^d \times \real^K)$,
    \begin{equation}
        \forall j \in \range{0}{J-1}, \qquad
        \label{eq:dg_conditioning}
        d_{g}\bigl(\op B_j \pi, \op B_j p\bigr)
        \leq
        L_{\op B}(R, \kappa_y) \, d_{g}\bigl(\pi, p\bigr).
    \end{equation}
\end{lemma}
\begin{proof}
    Let $\pi = \normal({\vect \tau}, \mat \Upsilon)$ and $p = \normal({\vect t}, \mat U)$.
    We use the shorthand notation $y^\dagger = y^\dagger_{j+1}$.
    It is well known that if $Z \sim \normal({\vect \tau}, \mat \Upsilon)$,
    for a vector $\vect \tau \in \real^{d+K}$ and a matrix $\mat \Upsilon \in \real^{(d+K)\times (d+K)}$,
    then the conditional distribution of~$Z_u$ given that $Z_y = y^\dagger$ is given by
    \begin{align}
        \label{eq:formula_conditional_gaussian}
        Z_u | Z_y = y^\dagger \sim \normal \bigl( {\vect \tau}_u + \mat \Upsilon_{uy} \mat \Upsilon_{yy}^{-1} (y^\dagger - {\vect \tau}_y), \mat \Upsilon_{uu} -  \mat \Upsilon_{uy} \mat \Upsilon_{yy}^{-1} \mat \Upsilon_{yu}\bigr).
    \end{align}
    Denoting $\op B_j \pi = \normal(\vect \lambda, \mat \Delta)$ and $\op B_j p = \normal(\vect \ell, \mat D)$,
    we have
    \begin{align*}
        \vect \lambda - \vect \ell
        &= {\vect \tau}_u - {\vect t}_u + \mat \Upsilon_{uy} \mat \Upsilon_{yy}^{-1} (y^\dagger - {\vect \tau}_y) - \mat U_{uy} \mat U_{yy}^{-1} (y^\dagger - {\vect t}_y) \\
        &= {\vect \tau}_u - {\vect t}_u
        + (\mat \Upsilon_{uy} - \mat U_{uy}) \mat \Upsilon_{yy}^{-1} (y^\dagger - {\vect \tau}_y)  \\
        &\qquad + \mat U_{uy} \mat U_{yy}^{-1} (\mat U_{yy} - \mat \Upsilon_{yy})\mat \Upsilon_{yy}^{-1} (y^\dagger - {\vect \tau}_y)
        - \mat U_{uy} \mat U_{yy}^{-1} (\vect \tau_y - \vect t_y).
    \end{align*}
    Since the 2-norm of any submatrix is bounded from above by the 2-norm of the matrix that contains it,
    we deduce
    \[
        \bigl\lvert \vect \lambda - \vect \ell \bigr\rvert
        \leq
        2 R^4 \left\lvert  {\vect \tau} - {\vect t} \right\rvert
        + 2 R^6 (R + \kappa_y) \norm{\mat \Upsilon - \mat U}.
    \]
    Similarly
    \(
        \norm{ \mat \Delta - \mat D }
        = \norm{ \mat \Upsilon_{uu} - \mat U_{uu} - \mat \Upsilon_{uy} \mat \Upsilon_{yy}^{-1} \mat \Upsilon_{yu} + \mat U_{uy} \mat U_{yy}^{-1} \mat U_{yu} } \leq 4 R^8 \norm{\mat \Upsilon - \mat U}.
    \)
    By Schur decomposition,
    it holds that
    \begin{align}
        \begin{pmatrix} \mat U_{uu} & \mat \mat U_{uy} \\ \mat U_{uy}^\t & \mat U_{yy} \end{pmatrix}
        = \begin{pmatrix} \mat I_{d} & \mat U_{uy} \mat U_{yy}^{-1} \\ 0 & \mat I_{K} \end{pmatrix}
        \begin{pmatrix} \mat U_{uu} -  \mat U_{uy} \mat U_{yy}^{-1} \mat U_{uy}^\t  & 0 \\ 0 & \mat U_{yy} \end{pmatrix}\begin{pmatrix} I_d & 0 \\ \mat U_{yy}^{-1} \mat U_{uy}^\t  & I_K \end{pmatrix}.
    \end{align}
    Since $\mat D = \mat U_{uu} -  \mat U_{uy} \mat U_{yy}^{-1} \mat U_{uy}^\t$,
    we have by the Courant--Fischer theorem
    \begingroup \footnotesize
    \begin{align*}
        \lambda_{\min}(\mat U)
        &= \min \left\{ \frac{x^\t \mat U x}{x^\t x} : x \in \real^{d+K} \setminus \{0\} \right\}
        \leq \min \left\{ \frac{x^\t \mat U x}{x^\t x} : x =
            \begin{pmatrix}
                u \\
                - \mat U_{yy}^{-1} \mat U_{uy}^\t u
            \end{pmatrix}
            \text{ and }
        u \in \real^{d} \setminus \{0\} \right\} \\
        &= \min \left\{ \frac{u^\t D u}{x^\t x} : x =
            \begin{pmatrix}
                u \\
                - \mat U_{yy}^{-1} \mat U_{uy}^\t u
            \end{pmatrix}
            \text{ and }
        u \in \real^{d} \setminus \{0\} \right\} \\
        &\leq
        \min \left\{ \frac{u^\t D u}{u^\t u} : u \in \real^{d} \setminus \{0\} \right\}
        = \lambda_{\min}(D).
    \end{align*}
    \endgroup
    Therefore~$\norm{D^{-1}} \leq \norm{\mat U^{-1}} \leq R^2$.
    Using \cref{lemma:dg_gaussians}, we then have
    \begin{align*}
        d_{g} (\op B_j \pi, \op B_j p)
        &\leq \sqrt{\op B_j \pi[g^2] + \op B_j p[g^2]}
        \left(3\norm{\mat D^{-1} \mat \Delta - I_d}_{\rm F} + \vecnorm{\vect \lambda - \vect \ell}_{\mat D}\right) \\
        &\leq \sqrt{\op B_j \pi[g^2] + \op B_j p[g^2]}
        \left(3\sqrt{d} \norm{\mat \Delta - \mat D} + \vecnorm{\vect \lambda - \vect \ell}\right) \left(1 + \norm{\mat D^{-1}}\right)  \\
        &\leq C(R, \kappa_y)
        \bigl(\norm{\mat \Upsilon - \mat U} + \vecnorm{{\vect \tau} - {\vect t}}\bigr),
    \end{align*}
    for an appropriate constant $C$ depending on $R$ and $\kappa_y$.
    By \cref{lemma:moment_bound} it holds that
    \[
        \norm{\mat \Upsilon - \mat U}
        + \bigl\lvert {\vect \tau} - {\vect t} \bigr\rvert
        \leq \left(\frac{3}{2} + \frac{1}{2}\vecnorm{{\vect \tau} + {\vect t}}\right) \, d_{g}(\pi, p),
    \]
    enabling us to conclude.
\end{proof}

\section{Technical Results for Corollaries \ref{theorem:main_theorem2} and \ref{theorem:main_theorem_gaussian_projection2}}
\label{appendix:C}

\begin{lemma}
    \label{lemma:new_lemma_revision}
    Suppose that $\Psi\colon \real^d \to \real^d$ and $\vect h \colon \real^d \to \real^K$ are functions
    such that~\cref{assumption:ensemble_kalman} is satisfied,
    and let $(\Psi_{n})_{n \in \nat}$ and $(h_{n})_{n \in \nat}$ be sequences of operators such that
    \[
        \Psi_{n} \xrightarrow[n \to \infty]{L^{\infty}(\real^d)} \Psi
        \qquad \text{ and } \qquad
        \vect h_{n} \xrightarrow[n \to \infty]{L^{\infty}(\real^d)} \vect h.
    \]
    Let $\op P_{n}$ and $\op Q_{n}$ denote the maps~\eqref{eq:Markov_kernel}
    and~\eqref{eq:definition_Q} associated with the functions~$\Psi_n$ and $\vect h_n$
    and assume that they too satisfy~\cref{assumption:ensemble_kalman}. Then
    \begin{align}
        \label{eq:limits_new_lemma}
        \sup_{\mu \in \mathcal P(\real^d)} d_g(\op P \mu, \op P_n \mu) \xrightarrow[n \to \infty]{} 0
        \qquad \text{and} \qquad
        \sup_{\mu \in \mathcal P_R(\real^d)} d_g(\op Q \mu, \op Q_n \mu) \xrightarrow[n \to \infty]{} 0,
    \end{align}
     for all $R \geq 1$ for the second statement.
\end{lemma}
\begin{proof}
    For all $\mu \in \mathcal P(\real^d)$,
    the probability measures $\op P \mu$ and $\op P_n\mu$ have Lebesgue densities.
    By the definition in equation~\eqref{eq:Markov_kernel} of~$\op P$ and~\cref{remark:total_variation},
    it holds that
    \[
        d_g(\op P \mu, \op P_n \mu)
        = \frac{1}{\sqrt{(2\pi)^d \det \mat \Sigma}}  \int_{\real^d} g(u)
        \left\lvert \int_{\real^d} \Delta_n(u, v) \, \mu(\d v) \right\rvert
        \, \d u,
    \]
    where
    \[
        \Delta_n(u, v) :=
         \exp \left( - \frac{1}{2} \vecnorm*{u - \vect \Psi(v)}_{\mat \Sigma}^2 \right)
        - \exp \left( - \frac{1}{2} \vecnorm*{u - \vect \Psi_n(v)}_{\mat \Sigma}^2 \right).
    \]
    Letting $\psi_s(v) = (1-s) \Psi(v) + s \Psi_n(v)$ and using the
    same reasoning as in~\cref{lemma:elementary_gaussians},
    we obtain that
    \begin{subequations}
        \begin{align}
            \notag
            \bigl\lvert \Delta_n(u, v) \bigr\rvert
        &\leq \bigl\lvert \Psi(v) - \Psi_n(v) \bigr\rvert_{\mat \Sigma}
        \int_{0}^{1} \bigl\lvert u - \psi_s \bigr\rvert_{\mat \Sigma} \, \exp \left(- \frac{1}{2} \bigl\lvert u - \psi_s \bigr\rvert_{\mat \Sigma}^2 \right) \, \d s \\
        \label{eq:second_ineq}
        &\leq
        \bigl\lvert \Psi(v) - \Psi_n(v) \bigr\rvert_{\mat \Sigma} \int_{0}^{1} \Bigl( \left\lvert u \right\rvert_{\mat \Sigma} + \left\lvert \psi_s \right\rvert_{\mat \Sigma} \Bigr) \, \exp \left(- \frac{1}{3} \bigl\lvert u \bigr\rvert_{\mat \Sigma}^2 + \bigl\lvert \psi_s \bigr\rvert_{\mat \Sigma}^2 \right) \, \d s \\
        \label{eq:third_ineq}
        &\leq C \bigl\lvert \Psi(v) - \Psi_n(v) \bigr\rvert_{\mat \Sigma} \exp \left(- \frac{1}{4} \bigl\lvert u \bigr\rvert_{\mat \Sigma}^2 \right).
        \end{align}
    \end{subequations}
    In~\eqref{eq:second_ineq},
    we used that by Young's inequality, it holds for all $\delta > 0$ that
    \[
        \vecnorm*{u - \psi_s}_{\mat \Sigma}^2
        \geq \frac{1}{1 + \delta} \vecnorm*{u}_{\mat \Sigma}^2 -
        \frac{1}{\delta} \vecnorm*{\psi_s}_{\mat \Sigma}^2.
    \]
    Then in~\eqref{eq:third_ineq},
    we used the bound $\norm{\psi_s}_{L^\infty} \leq \kappa_{\Psi}$,
    which holds for all $s \in [0, 1]$ by~\cref{assumpenum:assumption2}.
    The first limit in~\eqref{eq:limits_new_lemma} then follows immediately.

    For the second limit in~\eqref{eq:limits_new_lemma},
    fix $\mu \in \mathcal P_R(\real^d)$ and note that, by definition~\eqref{eq:definition_Q} of $\op Q$,
    \begin{align}
        \notag
        d_g(\op Q \mu, \op Q_n \mu)
        &\le
        \frac{1}{\sqrt{(2\pi)^K \det \mat \Gamma}}
        \int_{\real^d \times \real^K} \bigl(1 + \lvert u \rvert^2 + \lvert y \rvert^2\bigr)
           \left| \widetilde \Delta_n(u, y)\right| \, \mu(\d u) \, \d y \\
        \label{eq:close_constant_Q}
        &\leq
        \frac{1}{\sqrt{(2\pi)^K \det \mat \Gamma}}
        \int_{\real^d} \bigl(1 + \lvert u \rvert^2 \bigr)
        \left( \int_{\real^K} \bigl(1 + \lvert y \rvert^2 \bigr)
        \bigl\lvert \widetilde \Delta_n(u, y) \bigr\rvert \, \d y \right) \, \mu(\d u),
    \end{align}
    where we introduced
    \[
        \widetilde \Delta_n(u, y) := \exp \left( - \frac{1}{2} \bigl\lvert y - \vect h(u) \bigr\rvert_{\mat \Gamma}^2 \right) - \exp \left( - \frac{1}{2} \bigl\lvert y - \vect h_n(u) \bigr\rvert_{\mat \Gamma}^2 \right).
    \]
    Using the same strategy as above,
    we obtain that
    \[
        \forall (u, y) \in \real^d \times \real^K, \qquad
        \bigl\lvert \widetilde \Delta_n(u, y) \bigr\rvert
        \leq C \bigl\lvert h(u) - h_n(u) \bigr\rvert_{\mat \Gamma} \exp \left(- \frac{1}{4} \bigl\lvert y \bigr\rvert_{\mat \Gamma}^2 \right).
    \]
    Substituting in~\eqref{eq:close_constant_Q} gives the second limit in~\eqref{eq:limits_new_lemma}.
\end{proof}

\begin{proposition}
\label{prop:1}
    Suppose that $\Psi\colon \real^d \to \real^d$ and $\vect h \colon \real^d \to \real^K$ are functions taking constant values,
    and let $(\Psi_{n})_{n \in \nat}$ and $(h_{n})_{n \in \nat}$ be sequences of functions such that
    \[
        \Psi_{n} \xrightarrow[n \to \infty]{L^{\infty}(\real^d)} \Psi
        \qquad \text{ and } \qquad
        \vect h_{n} \xrightarrow[n \to \infty]{L^{\infty}(\real^d)} \vect h.
    \]
    Let $\op P_{n}$ and $\op Q_{n}$ denote the maps~\eqref{eq:Markov_kernel}
    and~\eqref{eq:definition_Q} (associated with the functions~$\Psi_n$ and $\vect h_n$)
    and assume that they too satisfy~\cref{assumption:ensemble_kalman}.
    Denote by $(\mu^n_{j})_{j \in \range{1}{J}}$ the true filtering distribution associated with
    functions $\Psi_n, \vect h_n$.
    Then,
    with the same notation as in~\cref{lemma:new_lemma_revision},
    it holds that
    \[
        \lim_{n \to \infty} \max_{j \in \range{0}{J-1}} d_g(\op G \op Q_n \op P_n \mu^n_j, \op Q_n \op P_n \mu^n_j) = 0.
    \]
\end{proposition}
\begin{proof}
    The conclusion follows from the stronger statement that
    \[
        \lim_{n \to \infty} \sup_{\mu \in \mathcal P(\real^d)} d_g(\op G \op Q_n \op P_n \mu, \op Q_n \op P_n \mu) = 0.
    \]
    Indeed, fix $\mu \in \mathcal P(\real^d)$.
    By the triangle inequality,
    we have that
    \begin{equation}
        \label{eq:3terms}
        d_g(\op G \op Q_n \op P_n \mu, \op Q_n \op P_n \mu)
        \leq
        d_g(\op G \op Q_n \op P_n \mu, \op G \op Q \op P \mu)
        + d_g(\op G \op Q \op P \mu, \op Q \op P \mu)
        +  d_g(\op Q \op P \mu, \op Q_n \op P_n \mu).
    \end{equation}
    Since $\Psi$ and $h$ are constant,
    the probability measure $\op Q \op P \mu$ is a Gaussian distribution independent of~$\mu$,
    and so the second term on the right-hand side is zero.
    For the third term,
    we have using~\cref{lemma:lipschitz_Q}, that
    \begin{align*}
        d_g(\op Q \op P \mu, \op Q_{n} \op P_{n} \mu)
        &\leq
        d_g(\op Q \op P \mu, \op Q_n \op P \mu)
        + d_g(\op Q_n \op P \mu, \op Q_n \op P_n \mu) \\
        &\leq
        d_g(\op Q \nu, \op Q_n \nu) + \Bigl(1 + \kappa_{\vect h_n}^2 + \trace(\mat \Gamma) \Bigr) d_g(\op P \mu, \op P_{n} \mu), \qquad \nu := \op P \mu.
    \end{align*}
    Noting that $\nu := \op P \mu$ is a fixed Gaussian measure independent of~$\mu$
    and using~\cref{lemma:new_lemma_revision},
    we deduce that both terms on the right-hand side tend to 0 in the limit~$n \to \infty$,
    uniformly in~$\mu \in \mathcal P(\real^d)$.
    It remains to show that the first term in~\eqref{eq:3terms} also tends to 0 as $n \to \infty$ uniformly in~$\mu \in \mathcal P(\real^d)$.
    In view of the moment bounds in~\cref{lemma:bound_on_QPmu},
    there exist $R \geq 1$ and~$N \in \nat$ such that
    \[
        \forall n \geq N, \qquad \forall \mu \in \mathcal P(\real^d), \qquad
        \op Q_n \op P_n \mu \in \mathcal P_R(\real^d)
        \qquad \text{and} \qquad
        \op Q \op P \mu \in \mathcal P_R(\real^d).
    \]
    Therefore, by the local Lipschitz continuity of~$\op G$ established in \cref{lemma:lipschitz_G},
    and the uniform-in-$\mu$ convergence to~0 of the third term in~\eqref{eq:3terms} that we already proved,
    it holds that
    \[
        \sup_{\mu \in \mathcal P(\real^d)} d_g(\op G \op Q_n \op P_n \mu, \op G \op Q \op P \mu)
        \xrightarrow[n \to \infty]{} 0,
    \]
    which concludes the proof.
\end{proof}


\end{document}

%% file: preamble.tex
\newif\ifsiamart
\makeatletter%

\@ifclassloaded{siamart171218}%
  {\siamarttrue}%
  {\siamartfalse}%
\makeatother%

\usepackage{silence}
\usepackage[utf8]{inputenc}
\usepackage[margin=1in]{geometry}
\usepackage{xcolor}
\usepackage{graphicx}
\usepackage{bm}
\usepackage{bbm}
\usepackage{amsmath}
\usepackage{amssymb}
\usepackage{stmaryrd}
\usepackage{amsfonts}
\usepackage{mathrsfs}
\usepackage{mathtools}
\usepackage{xparse}
\usepackage{epstopdf}
\usepackage{enumitem}

\ifsiamart\else
\usepackage{amsthm}
\usepackage{hyperref}
\hypersetup{%
    linktocpage=true,
    colorlinks=true,
    citecolor=red,
    linkcolor=blue,
}
\usepackage[nameinlink,capitalise]{cleveref}
\newcommand{\email}[1]{\href{mailto:#1}{#1}}
\newenvironment{keywords}{\noindent \textbf{Keywords.}}{}
\newenvironment{AMS}{\noindent \textbf{AMS subject classification.}}{}
\fi

\definecolor{darkred}{rgb}{.7,0,0}
\definecolor{darkgreen}{rgb}{.1,.7,0}

\usepackage{setspace}
\DeclarePairedDelimiter\matnorm{\lVert}{\rVert}
\DeclarePairedDelimiter\vecnorm{\lvert}{\rvert}
\DeclarePairedDelimiter\abs{\lvert}{\rvert}
\DeclarePairedDelimiter\norm{\lVert}{\rVert}

\DeclareMathOperator{\e}{e}

\DeclareMathOperator*{\trace}{tr}
\DeclareMathOperator*{\argmin}{arg\,min}

\DeclareMathOperator{\ran}{Im}

\renewcommand{\t}{\mathsf T}

\newcommand{\kl}[2]{\mathrm {KL} (#1 \| #2)}

\newcommand{\placeholder}{\mathord{\color{black!33}\bullet}}%
\newcommand{\ind}{\perp\!\!\!\!\perp}
\newcommand{\expect}{\mathbf{E}}

\newcommand{\range}[2]{\llbracket #1, #2 \rrbracket}
\newcommand{\mat}{\mathit}
\newcommand{\op}{\mathsf}
\newcommand{\normal}{\mathcal{N}}
\newcommand{\nat}{\mathbf N}

\newcommand{\real}{\mathbf R}

\newcommand{\vect}[1]{#1}

\renewcommand{\d}{\mathrm d}

\renewcommand{\leq}{\leqslant}
\renewcommand{\geq}{\geqslant}
\renewcommand{\le}{\leqslant}
\renewcommand{\ge}{\geqslant}
\newcommand{\cdott}{\placeholder}

\newcommand{\mean}{\vect m}
\newcommand{\pmean}{\widehat{m}}
\newcommand{\Cov}{\mat C}
\newcommand{\pCov}{\widehat{\mat C}}
\newcommand{\hmu}{\widehat{\mu}}
\newcommand{\yd}{{y}^\dag}
\newcommand{\hv}{\widehat{v}}
\newcommand{\hy}{\widehat{y}}

\ifsiamart
\newtheorem{remark}[theorem]{Remark}
\newtheorem{assumption}{Assumption}

\else
\theoremstyle{plain}

\newtheorem{theorem}{Theorem}
\newtheorem{lemma}[theorem]{Lemma}
\newtheorem{corollary}[theorem]{Corollary}
\newtheorem{proposition}[theorem]{Proposition}

\newtheorem{remark}[theorem]{Remark}
\newtheorem{definition}[theorem]{Definition}

\newtheorem{assumption}{Assumption}

\numberwithin{equation}{section}
\crefname{lemma}{Lemma}{Lemmas}
\crefname{remark}{Remark}{Remarks}
\crefname{assumption}{Assumption}{Assumptions}
\crefname{proposition}{Proposition}{Propositions}
\crefname{equation}{}{}
\Crefname{equation}{Equation}{Equations}
\fi

\crefname{section}{Section}{Sections}
\crefname{subsection}{Subsection}{Subsections}

\usepackage{tikz}
\usepackage{tikz-cd}
\usepackage{pgfplotstable}
\pgfplotsset{compat=1.14}
\usetikzlibrary{patterns}
\usetikzlibrary{calc}
\usetikzlibrary{angles}
\usetikzlibrary{quotes}
\usetikzlibrary{external}

\usepackage[url=true,backend=biber,style=trad-abbrv,url=false,doi=false,isbn=false]{biblatex}
\DeclareFieldFormat{volume}{volume \textbf{#1}}
\DeclareFieldFormat[article]{volume}{\textbf{#1}}

\ifsiamart\else
\let\oldparagraph=\paragraph
\renewcommand\paragraph[1]{\oldparagraph{#1.}}
\fi

\newlist{assumpenum}{enumerate}{5}
\setlist[assumpenum]{font={\bfseries}}
\crefname{assumpenumi}{Assumption}{Assumptions}